\documentclass{article}
\usepackage{amsthm,amsmath,amssymb}
\usepackage{color}

\RequirePackage[colorlinks,citecolor=blue,urlcolor=blue]{hyperref}

\def\bR{\mathbb R}
\def\bE{\mathbb E}
\def\bW{\mathbb W}

\newtheorem{Theorem}{Theorem}

\newtheorem{lemma}{Lemma}
\newtheorem{Corollary}{Corollary}
\newtheorem{proposition}{Proposition}
\newtheorem{Assumption}{Assumption}
\newtheorem{example}{Example}

\def\argmin{\mathop{\rm arg\,min}}

\def\rank{{\rm rank}}

\def\supp{{\rm supp}}

\def\Sum{\overset{n}{\underset{i=1}{\sum}}}

\begin{document}
%\begin{frontmatter}

\title{Non-asymptotic approach to varying coefficient model}
\author{Olga Klopp\footnote{MODAL'X, University Paris Ouest Nanterre (kloppolga@math.cnrs.fr)} \hskip 0.25 cm and Marianna Pensky \footnote{Department of Mathematics, University of Central Florida (marianna.pensky@ucf.edu)} }
\date{ }
%\address{}
%\email{olga.klopp@ensae.fr}
\maketitle
%\runtitle{Matrix completion with unknown variance}

% indicate corresponding author with \corref{}
% \author{\fnms{John} \snm{Smith}\thanksref{t2}\corref{}\ead[label=e1]{smith@foo.com}\ead[label=e2,url]{www.foo.com}}
% \thankstext{t2}{Thanks to somebody} 
% \address{line 1\\ line 2\\ \printead{e1}\\ \printead{e2}}
%\begin{aug}
%\author{\fnms{Olga} \snm{Klopp}\ead[label=e1]{olga.klopp@ensae.fr}}
%\address{Laboratoire de Statistique, CREST and University Paris Dauphine\\CREST 3, Av. Pierre Larousse 92240 Malakoff France\\\printead{e1}}
%\and
%\author{\fnms{???} \snm{???}\ead[label=e2]{???}}
%\address{\printead{e2}}

%\runauthor{Olga Klopp}
%\end{aug}
%\begin{abstract}
%\end{abstract}

%\begin{keyword}[class=AMS]
%\kwd{62J99, 62H12, 60B20, 60G15}
%%\kwd{}
%%\kwd[; secondary ]{}
%\end{keyword}
%
%\begin{keyword}
%\kwd{matrix completion, low rank matrix estimation}
%\kwd{ unknown variance, statistical learning}
%\end{keyword}

% history:
% \received{\smonth{1} \syear{0000}}

%\tableofcontents

%\end{frontmatter}
%\section{Introduction}\label{introduction}
\begin{abstract}
In the present paper we consider the varying coefficient model   which represents a useful tool for exploring   dynamic patterns in many applications. Existing methods typically provide asymptotic  evaluation of precision of estimation procedures under the assumption that the number of observations tends to infinity. In practical applications, however, only a finite number of measurements are available. In the present paper we focus on a non-asymptotic approach to the problem. We propose a novel estimation procedure which is based on recent developments in matrix estimation.
%and  has a guaranteed performance for a finite number of observations.
 In particular, for our estimator, we obtain upper bounds
for the  mean squared   and the pointwise estimation errors. The obtained oracle inequalities are non-asymptotic and hold for finite sample size.
%These upper bounds  hold for finite sample sizes and various relationships between the 
%parameters in the model.
\end{abstract}

%%%%%%%%%%%%%%%%%%%%%%%%%%%%%%%%%%%%%%%%%%%%%%%%%%%%%%%%%%%%%%%%%%%%%%%%%%%%%%%%%%%%%%%%%%%%%%%%%%%%%%%%%%%%%%%%%%%%%%%%%%%%%%%%
%%%%%%%%%%%%%%%%%%%%%%%%%%%%%%%%%%%%%%%%%%%%%%%%%%%%%%%%%%%%%%%%%%%%%%%%%%%%%%%%%%%%%%%%%%%%%%%%%%%%%%%%%%%%%%%%%%%%%%%%%%%%%%%%%%%

\section{Introduction} 
\label{sec:introduction}

In the present paper we consider the varying coefficient model   which represents a useful tool for exploring   dynamic patterns   in economics, epidemiology, ecology, etc. This model can be viewed as a natural extension of the classical linear regression model and allows parameters that are   constant in regression model to evolve with certain characteristics of the system such as  time or  age in epidemiological studies.

The varying coefficient models  were introduced by Cleveland, Grosse and Shyu \cite{cleveland_VCM} and 
Hastie and Tibshirani \cite {hastie_VCM} and have been extensively studied in the past 15 years.  
 The estimation procedures for varying coefficient model are e.g. based on 
the kernel-local polynomial smoothing (see e.g. \cite{wu,hoover,fan_1999,kauermann}), the  polynomial spline (see e.g. \cite{huang_2002,huang_2004,huang_shen}), the smoothing spline (see e.g. \cite{hastie_VCM,hoover,chiang}).
 More recently e.g. Wang et al \cite{wang} proposed a new procedure based on a local rank estimator; Kai et al \cite{kai} introduced a semi-parametric quantile regression procedure and studied an effective variable selection procedure;  Lian \cite{lian1} developed a penalization based approach for both variable selection and constant coefficient identification in a consistent framework. % Of course, this list does not pretend to be exhaustive. 
 For more detailed discussions of the existing methods and possible applications, we refer to the very interesting survey of Fan and Zhang \cite{fan}. 
 
 Existing methods typically provide asymptotic evaluation of precision of estimation procedures under the assumption that the number of observations tends to infinity. In practical applications, however, only a finite number of measurements are available. In the present paper, we focus on a non-asymptotic approach to the problem. We propose a novel estimation procedure which is based on recent developments in matrix estimation, in particular,  matrix completion. In the matrix completion problem, one observes a small set of entries of a matrix and needs to estimate the remaining entries using these data.   A standard assumption that allows such  completion to be successful is  that the unknown matrix has 
 low rank or has approximately low rank. 
The matrix completion problem has attracted a considerable attention   in the past few years (see, e.g., \cite{candes-recht-exact,keshavan-few, Koltchinskii-Tsybakov, wainwright-weighted, klopp-general}). The most popular methods for matrix completion are based on nuclear-norm minimization which we adapt in the present paper.

%%%%%%%%%%%%%%%%%%%%%%%%%%%%%%%%%%%%%%%%%%%%%%%%%%%%%%%%%%%%%%%%%%%%%%%%%%%%%%%%%%%%%%%%%%%%%%%%%%%%%%%%%%%%%%%%%%%%

%\subsection{Problem set-up and estimation procedure description}

\subsection{Formulation of the problem}

Let $(W_i,t_i,Y_i)$, $i=1,\dots,n$ be sampled independently from the   varying coefficient model 
\begin{equation}\label{vcm}
Y=W^{T}f(t)+\sigma \xi.
\end{equation} 
Here, $W\in\mathbb R^{p}$ are random vectors of predictors, $f(\cdot)=\left (f_1(\cdot),\dots,f_p(\cdot)\right )^{T}$ is an unknown vector-valued function of regression coefficients and $t\in[0,1]$ is a random variable independent of $W$. Let $\mu$ denote its distribution. 
The   noise variable  $\xi$ is independent of $W$ and $t$ and is  such that 
$\bE(\xi)=0$ and $\bE(\xi^{2})=1$, $\sigma>0$ denotes the   noise level.

The goal is to estimate the vector function $f(\cdot)$ on the basis of observations  $(W_i,t_i,Y_i)$, $i=1,\dots,n$.
Our estimation method  is based on the approximation of the unknown functions $f_i(t)$ using a basis expansion. 
This approximation generates the coordinate matrix $A_0$. 
%In the above model, some of the components of vector function $f$ are constant (i.e. coefficients of the dictionary W are not time varying),
%the rest of them are time dependent.
 In the above model, some of the components of vector function $f$ are constant. The larger the part of the constant  regression coefficients, the smaller the  rank of the coordinate matrix $A_0$ (the rank of matrix  $A_0$  does not exceed the number of time-varying components of vector $f(\cdot)$ by more than one).
We suppose that the first element of this basis is just a constant function on $[0,1]$ (indeed, this is true for vast majority of bases on a finite interval). In this case,   if the component $f_i(\cdot)$ is   constant, then, it has only one non-zero coefficient in its expansion over the basis.
%On the other hand,  with a good choice of the basis, we should have only a small number of large coefficients in each row of $A_0$ and all the other coefficients should be equal to zero or small. For example, if functions $f_i(\cdot)$ are smooth and we choose the Fourier basis, $f_i(\cdot)$ can be represented by a small number of large coefficients. 
  %so that the number of large elements in each row of matrix $A_0$ will be limited, and the rest of the elements will be small or vanish.
This suggest the idea  to take into account the number of constant regression coefficients using the rank of the coordinate matrix $A_0$.
%Second, to use the sizes of the $l_1-$norms of the rows of $A_0$ to remove insignificant coefficients in the expansions of 
%$f_i$.
% This idea leas to a new procedure for recovering $A_0$ which combines a nuclear norm penalization  that is now standard in matrix completion, with the group LASSO-type  penalization that promotes small $l_1-$norm of the rows. 
% 
% A similar idea has been used  in a recent preprint \cite{chen} for estimating multiple predictive functions from a dictionary of basis functions. Assuming that each predictive function can be estimated in the form of a linear combination of the basis functions and that the coefficient matrix admits a sparse low-rank structure, Chen and Ye in \cite{chen} introduced a procedure based on simultaneous trace norm and $l_1-$norm penalisation. We use a different penalisation (the maximum of the $l_1-$norms of the rows) and our framework is quite different. 
% 

Our procedure involves  estimating $A_0$ using nuclear-norm penalization % minimization
 which is now a well-established proxy for rank penalization in the compressed sensing literature. 
Subsequently, the estimator of the coordinate matrix is plugged into the expansion yielding the estimator $\hat f(\cdot)=\left (\hat f_1(\cdot),\dots,\hat f_p(\cdot)\right )^{T}$ of the vector function $f(t)$. For this estimator  we obtain upper bounds on the  mean squared error  $\dfrac{1}{p}\underset{i=1}{\overset{p}{\Sigma}}\Vert \hat f_i-f_i \Vert_{L_2(d\mu)}^{2}$   and on the pointwise estimation error  
$\dfrac{1}{p}\underset{i=1}{\overset{p}{\Sigma}}\vert \hat f_i(t)-f_i(t) \vert$ for any $t\in\supp(\mu)$ (Corollary \ref{cor_1}). These oracle inequalities are non-asymptotic and hold for finite values of $p$ and $n$. The results in this paper 
%\cblue{are derived under the assumptions that the measurement points, the dictionary and the noise are random  and  hold with high probability.}
 concern random measurements and random noise and so they hold with high probability.

%%%%%%%%%%%%%%%%%%%%%%%%%%%%%%%%%%%%%%%%%%%%%%%%%%%%%%%%%%%%%%%%%%%%%%%%%%%%%%%%%%%%%%%%%%%%%%%%%%%%%%%%%%%%%%%%%%%%%%%%%%%%%%

 \subsection{Layout of the paper}

   The remainder of this paper is organized as follows. In Section \ref{sec:notations}  we introduce   notations used throughout the paper. 
In Section \ref{estimation method},  we describe in details our estimation method, give examples of the possible choices of the basis 
(Section \ref{basis}) and  introduce an   estimator for the coordinate matrix $A_0$ (Section \ref{estimator}). 
Section \ref{main}  presents the main results of the paper. In particular, Theorems \ref{thm2} and \ref{thm3} in Section \ref{main}
establish  upper bounds for estimation error of  the  coordinate matrix $A_0$ measured in Frobenius norm.
% prediction error measured in Frobenius norm of the estimator  of the coordinate matrix. 
Corollary \ref{cor_1} provides non-asymptotic upper bounds for the mean squared and pointwise risks of the estimator 
of the vector function $f$.
Section \ref{orth}  considers an important particular case  of the  orthogonal dictionary.

%%%%%%%%%%%%%%%%%%%%%%%%%%%%%%%%%%%%%%%%%%%%%%%%%%%%%%%%%%%%%%%%%%%%%%%%%%%%%%%%%%%%%%%%%%%%%%%%%%%%%%%%%%%%%%%%%%%%%%%%%%%%%%

% \subsection{Notation} 
\subsection{Notations} 
\label{sec:notations} 
We provide a brief summary of the notation used throughout this paper. 
Let $A,B$ be matrices in $\mathbb{R}^{p\times l}$, $\mu$ be a probability distribution on $(0,1)$  
and $\psi(\cdot)$ be a vector-valued function.
\begin{itemize}
\item 
For any vector  $\eta \in \mathbb{R}^{p}$, we denote  the standard $l_1$ and $l_2$ vector norms by 
$ \left\Vert \eta \right\Vert_{1}$ and $ \left\Vert \eta \right\Vert_{2}$, respectively.
\item 
$\left\Vert\cdot \right\Vert_{L_2\left (d\mu\right )}$ and $\left\langle\cdot\,,\cdot\right\rangle_{L_2\left (d\mu\right )}$ 
are the norm and the scalar product in the space $L_2\left ((0,1),d\mu\right )$.
\item 
For $\psi(\cdot)=\left (\psi_1(\cdot),\dots,\psi_p(\cdot)\right )^{T}$, we set 
$ \left\Vert \psi (\cdot)\right\Vert_\infty=\underset{i=1,\dots,p}{\max}\;\underset{t\in \supp (\mu)}{\sup}\left \vert \psi_i(t)\right \vert$ 
and $\left\Vert \psi (\cdot)\right\Vert_{L_2(d\mu)}=\underset{1\leq i \leq p}{\max}\left\Vert \psi_i \right\Vert_{L_2(d\mu)}$
\item 
We define the \textit{scalar product of matrices}
$\langle A,B\rangle =\mathrm{tr}(A^{T}B)$ where $\mathrm{tr}(\cdot)$ denotes the trace of a square  matrix.
\item Let 
%For $0<q<\infty$ the \textit{Schatten-q (quasi-)norm} of the matrix $A$  is defined by
\begin{equation*}
\Vert A\Vert_*=\underset{j=1}{\overset{\min(p,l)}{\Sigma}}\sigma_j(A)\quad\text{and}\quad \Vert A\Vert_2=\left (\underset{j=1}{\overset{\min(p,l)}{\Sigma}}\sigma^{2}_j(A)\right )^{1/2}
\end{equation*}
be respectively the \textit{trace} and \textit{Frobenius} norms of the matrix $A$. Here  $(\sigma_j(A))_j$ are the singular values of $A$ ordered decreasingly.
\item 
Let $\Vert A\Vert=\sigma_1(A)$. %, $\left\Vert A\right\Vert_{\sup}=\underset{i,j}{\max}\mid a_{ij}\mid$ and 
%$ \left\Vert A\right\Vert_*=\underset{1\leq i\leq p}{\max}\;\underset{j=1}{\overset{l}{\Sigma}}\left \vert a_{ij}\right \vert$ where $A=(a_{ij})$.
%
 \item  For any numbers, $a$ and $b$, denote $a \vee b = \max(a,b)$ and $a \wedge b = \min(a,b)$.
 
%\item Let $d=p+l$, $\Omega=\bE \left (W\,W^{T}\right )$ and $M=\mathrm{tr}(\Omega)\vee \left (l\,\omega_{\max}\right )$.

\item Denote  the $k\times k$ identity  matrix by $\mathbb{I}_{k}$.
\item Let $(s-1)$ denote the number of non-constant $f_i(\cdot)$.
\item 
In what follows, we use the symbol $C$ for a generic positive
constant, which is independent of $n$, $p$, $s$ and $l$, and may take different
values at different places.
\end{itemize}

%%%%%%%%%%%%%%%%%%%%%%%%%%%%%%%%%%%%%%%%%%%%%%%%%%%%%%%%%%%%%%%%%%%%%%%%%%%%%%%%%%%%%%%%%%%%%%%%%%%%%%%%%%%%%%%%%%%%%%%%%%%

\section{Estimation method}\label{estimation method}

The first step of our estimation method is the approximation of the unknown functions $f_i(t)$ by expanding them % using their expansions 
over an appropriate  basis. This approximation generates the coordinate matrix $A_0$. Matrix $A_0$ is estimated using penalized risk minimization. 
The estimator of the coordinate matrix is plugged into the expansion  yielding the estimator of the vector function $f$. 
%The penalty used in the paper is the combination of the   nuclear norm   and the maximum of the $l_1-$norms of the rows.

\subsection{Basis expansion}\label{basis}

 Let $(\phi_i(\cdot))_{i=1,\dots,\infty}$ be an orthonormal basis in $L_2\left ((0,1),d\mu\right )$, 
$l\in \mathbb N$ and $\phi(\cdot)=\left (\phi_1(\cdot),\dots,\phi_l(\cdot)\right )^{T}$. We assume that basis functions satisfy the following condition: there exists $c_\phi < \infty$ such that   
\begin{equation} \label{eq:basis_assum}
\left\Vert \phi^{T}(t)\right\Vert^{2}_2 = \sum_{j=1}^l |\phi_j(t)|^2 \leq c_\phi^2\, l,
\end{equation}
for any $l \geq 1$ and any $t \in [0,1]$. Note that this condition is satisfied for most of the usual bases.

 We introduce  the coordinate matrix $A_0 \in\mathbb R^{p\times l}$ with elements 
  \begin{equation*}
  a^{0}_{kj}=\left\langle f_k,\phi_j\right\rangle_{L_2(d\mu)},\ \ k=1, \cdots, p,\, j=1, \cdots, l.
  \end{equation*}
 For each $k=1,\dots,p$, we have 
 \begin{equation}\label{eq:expansions}
 f_k(t)=\underset{j=1}{\overset{l}{\Sigma}}a^{0}_{kj}\phi_j(t)+\rho^{(l)}_k(t).
 \end{equation}

Denote the remainder by $\rho^{(l)}(\cdot)=(\rho_1^{(l)}(\cdot),\dots,\rho_p^{(l)}(\cdot))^{T}$.
% be the bias vector  and $c_{\phi} =\left\Vert \phi(\cdot) \right\Vert_\infty$.  We suppose that $c_{\phi}<\infty$.  
We assume that the basis $(\phi_i(\cdot))_{i=1,\dots,\infty}$   guarantees   good approximation of $f_k$ by 
$\underset{j=1}{\overset{l}{\Sigma}}a^{0}_{kj}\phi_j(t)$, that is,

\begin{Assumption}\label{approximation_sup} 
 We assume that the basis satisfies condition \eqref{eq:basis_assum} and that 
there exists a  positive constant $b$ such that, for any $l\geq 1$  
\begin{eqnarray}
\left\Vert \rho^{(l)}(\cdot)\right\Vert_{\infty}& \leq &  b\ l^{-\gamma},\quad   \gamma > 0.
\end{eqnarray}
 \end{Assumption}  
 Often approximation in $L_2-$norm gives better rates of convergence. In order to get  upper bounds on the  mean squared error we will use the following additional assumption:
  \begin{Assumption}\label{approximation_L2}  There exist $b_1>0$ such that, for any $l\geq 1$  
  $$  \left\Vert \rho^{(l)}(\cdot)\right\Vert_{L_{2}(d\mu)}\leq b_1\,l^{-(\gamma+1/2)},\quad 
   \gamma > 0.$$
  \end{Assumption}
 
Let us give few examples of possible choices of the basis.

\begin{example} \label{example1}
{\rm Assume that  $d\mu=g(t)\,dt$ and  
 function $g$ is bounded away from zero and infinity, i.e. there exist absolute constants $g_1$ and $g_2$ such that for any $t\in \supp (\mu)$ 
 \begin{equation}  \label{g_condition}
g_1 \leq g(t) \leq g_2, \quad 0 < g_1 < g_2 < \infty.
 \end{equation}
 Denote $\widetilde{\phi}_j (t) = e^{2\,i\,\pi\, j\,t}$, $j\in\mathbb{Z}$, the standard Fourier basis  of $L_2\left ((0,1) \right )$.
Then, it is easy to check that $\phi_j (t) = \widetilde{\phi}_j (t)/ \sqrt{g(t)}$, $j\in\mathbb{Z}$, 
is an orthonormal basis of $L_2\left ((0,1),g\right )$. 
Moreover, condition \eqref{eq:basis_assum} holds with $c_\phi^2 = g_1^{-1}$.

For $\gamma>0$, consider the Sobolev space $\bW_{\gamma }(0,1)$ of functions $F \in  L_2 (0,1)$  with the norm
$\|F\|_{\bW_\gamma}^2 = \int_{-\infty}^\infty |\omega|^{2 \gamma+1} |\hat{F} (\omega) |^2 d \omega$ where $\hat{F}(\omega)$
is the Fourier transform of $F$. Then, by Theorems 9.1 and 9.2 of \cite{mallat}, one has
\begin{equation} \label{Fourier_Mallat}
\sum_{j=-\infty}^\infty |j|^{2 \gamma+1} | \langle F, \widetilde{\phi}_j  \rangle |^2 \leq C_\gamma \| F \|_{\bW_\gamma}^2,  
\end{equation}
where $C_\gamma$ is an absolute constant which depends on $\gamma$ only. }
{\rm
% Assume that the functions $f_k$ are such that for some $A < \infty$
Assume that for some $A < \infty$ the functions $f_k$ belong to a Sobolev ball of radius $A$, i.e.  
 \begin{equation}  \label{eq:fk_Fourier}
\max_{k=1, \cdots, p}\ \left\| f_k (\cdot) \sqrt{g(\cdot)} \right\|_{\bW_\gamma} \leq A, \quad \gamma > 0.
 \end{equation}
Let $l=2N+1$, so that 
 \begin{equation*}  
f_k (t) =   \sum_{j=-N}^N a^{0}_{kj} \phi_j(t), \quad 
\rho^{(l)}_k(t) = \sum_{|j|>N}  a^{0}_{kj} \phi_j(t),
 \end{equation*}
where $a^{0}_{kj} =  \langle f_k (t) \sqrt{g(t)}, \widetilde{\phi}_j (t) \rangle$. 
Then, it follows from equations \eqref{g_condition},  \eqref{Fourier_Mallat} and \eqref{eq:fk_Fourier} that 
\begin{equation*} 
\begin{split}
\left\| \rho^{(l)}(\cdot)\right\|_{\infty}^2   & \leq g_1^{-1} \ \left[ \sum_{|j| > N} |j|^{-2\gamma-1} \right] \ \left[ \max_{k=1, \cdots, p}\ 
 \sum_{|j| > N} |j|^{2\gamma+1}\, |a^{0}_{kj}|^2 \right] \\
& \leq \frac{A^2 C_\gamma}{g_1} \ \sum_{|j| > N} |j|^{-2\gamma-1}
  \leq \frac{A^2 C_\gamma}{2\,g_1\,\gamma \,N^{2\gamma}}
\end{split}
 \end{equation*}
 where $N=(l-1)/2$ 
and
 \begin{equation*}
\begin{split}
\left\| \rho^{(l)}(\cdot)\right\|_{L_2(g)}^2  \leq  N^{-(2 \gamma+1)}  A^2 C_\gamma, 
\end{split}
\end{equation*} 
so that  Assumptions \ref{approximation_sup} and \ref{approximation_L2} hold.}
\end{example} 
%
%%%%%%%%%%%%%%%%%%%%%%%%%%%%%%%%%%%%%%%%%%%%%%%%%%%%%%%%%%%%%%%%%%%%%%%%%%%%%%%%%%%%%%%%%%%%%%%%%%%%%%
%
\begin{example}  \label{example2}
{\rm
Consider a wavelet $\psi$ with a bounded support of length $C_\psi$ and with $\gamma^*$ 
vanishing moments and choose $l = 2^H$ where $H$ is a positive integer. 
Construct  a periodic wavelet basis $\psi_{h,i} (t)$, $h=-1, \cdots, J-1$, $i=0, \cdots, 2^h-1$,
with $\psi_{-1,0} (t) = 1$ and  $\psi_{h,i} (t)  = 2^{h/2} \psi(2^h t - i)$  for $h \geq 0$. 
As in Example~\ref{example1}, set  $\phi_j (t) =  \phi_{h, i} (t) = \psi_{h,i} (t)/\sqrt{g(t)}$ 
where $j = 2^h + i +1$. 
Note that condition \eqref{eq:basis_assum} holds in this case with $c_\phi^2 = g_1^{-1} C_\psi \| \psi \|^2_\infty$.

Then, each function $f_k(t)$ can be expanded into a wavelet series 
 \begin{equation*}  
f_k (t) =   \sum_{h=-1}^{H-1} \sum_{i=0}^{2^h-1}   a^{0}_{k, h, i}\, \phi_{h, i} (t), \quad 
\rho^{(l)}_k(t) = \sum_{h=H}^\infty \sum_{i=0}^{2^h-1}  a^{0}_{k,h,i}\, \phi_{h, i} (t), 
 \end{equation*}
where $a^{0}_{k,h,i} =  \langle f_k (\cdot) \sqrt{g(\cdot)}, \psi_{h,i} (\cdot) \rangle$.

Theorem 9.4 of \cite{mallat} states that for $F \in \bW_\gamma (0,1)$ one has
\begin{equation*} %\label{wavelet_Mallat}
\sum_{h=-1}^\infty  2^{h (2\gamma+1)} \sum_{i=0}^{2^h-1}   | \langle F, \psi_{h,i}  \rangle |^2 \leq C_\gamma \| F \|_{\bW_\gamma}^2,  
\end{equation*}
where $C_\gamma$ is an absolute constant which depends on $\gamma$ only, provided $\gamma < \gamma^*$. 
Then, under assumptions \eqref{g_condition} and  \eqref{eq:fk_Fourier}, as in Example \ref{example1},
Assumption \ref{approximation_sup} holds. For example, 
recalling that $H =\log_2 l$ and that length of support of $\psi$ is bounded by $C_\psi$, obtain
\begin{equation*} 
\begin{split}
  \left\| \rho^{(l)}(\cdot)\right\|_{L_2(g)}^2 & \leq  2^{-H (2 \gamma+1)}   \max_{k=1, \cdots, p}\ 
\sum_{h=H}^\infty  2^{h (2\gamma+1)}\ \sum_{i=0}^{2^h-1} |a^{0}_{k,h,i}|^2 
\leq  A^2 C_\gamma  l^{-(2\gamma+1)},\\
% \\  & \leq  A^2 C_\gamma  l^{-2\gamma},\quad H =\log_2 l.
%
\left\| \rho^{(l)}(\cdot)\right\|^2_\infty  & \leq 
 (2 \gamma g_1)^{-1} C_\psi\, \| \psi \|^2_\infty\  2^{-2H \gamma}   \max_{k=1, \cdots, p}\  
\sum_{h=-1}^\infty 2^{h (2\gamma+1)}\ \sum_{i=0}^{2^h-1} |a^{0}_{k,h,i}|^2\\
 & \leq  A^2 (2 \gamma g_1)^{-1} C_\psi\, \| \psi \|^2_\infty\  l^{- 2\gamma}, 
\end{split}
 \end{equation*}
where $\|\psi\|_\infty = \sup_t |\psi(t)|.$
}

 \end{example}
%
%
%%%%%%%%%%%%%%%%%%%%%%%%%%%%%%%%%%%%%%%%%%%%%%%%%%%%%%%%%%%%%%%%%%%%%%%%%%%%%%%%%%%%%%%%%%%%%%%%%%%%%%%%%%
%
%
% % % % % % % % % % % % % % % % % % % % % % % % % % % % % % % % % % % % % % % % % % % % % % % % % % % % % % % %
\begin{example}  \label{example3}
{\rm Suppose that $f_i(t)$ belong to a finite $k-$dimensional sub-space of $L_2\left ((0,1),d\mu\right )$. For example, $f_i(t)$ are polynomials of degree less than $k$. Then, choosing $l=k$ and an orthonormal basis in this sub-space, we have trivially $\rho^{(l)}(\cdot)=0$.
}
\end{example}
%
%
% % % % % % % % % % % % % % % % % % % % % % % % % % % % % % % % % % % % % % % % % % % % % % % % % % % % % % % % % % % % % % % %

\subsection{Estimation of the coordinate matrix}\label{estimator}

  Denoting $X=W\phi^{T}(t)$, we can rewrite \eqref{vcm} in the following form
 \begin{equation}\label{model}
 Y =\mathrm{tr}\left (A_0 X^{T}\right )+W^{T}\rho^{(l)}(t)+\sigma\xi.
 \end{equation} 
  We suppose that some of the functions $f_i(\cdot)$ are constant and let $(s-1)$ denote the number of non-constant $f_i(\cdot)$. This parameter, $s$, plays an important role in what follows. Note that $\rank\left (A_0\right )\leq s$.
%Since $\rank\left (A_0\right )\leq s$, parameter, $s$, plays an important role in the proposed estimation procedure.

Using observations $(Y_i,X_i)$ we define the following estimator of $A_0$:
 \begin{equation}\label{estimator_orth}
 \hat{A}=\argmin\left \{\dfrac{1}{n}\Sum \left (Y_i-\left\langle X_i,A\right\rangle\right )^{2}+\lambda\,\left\Vert A\right\Vert_{*}\right \},
 \end{equation}
 where $\lambda$ is the regularization parameter. This  penalization, using the trace-norm,  is now quite standard in matrix completion problem 
and  allows one to recover a matrix from under-sampled measurements. % With a good choice of the basis, for each $f_k$, we should have only a small number of large coefficients $a^{0}_{kj}$ and all the other coefficients equal to zero or small enough, i.e. matrix $A_0$ should have small $ \left\Vert\cdot \right\Vert_*$ - norm. The second penalization, using the maximum of the $l_1$-norms of the rows promotes the solutions with a small number of non-zero coefficients in each row.  
 
 Using estimator \eqref{estimator_orth} of the coordinate matrix $A_0$, we  recover $f(t)$ as 
$$\hat{f}(t)=\hat{A}\phi(t).$$

% % % % % % % % % % % % % % % % % % % % % % % % % % % % % % % % % % % % % % % % % % % % % % % % % % % % % % % % % % % %

 \subsection{Assumptions about the dictionary and the noise}

We assume that the vectors $W_i$ are i.i.d copies of a random vector $W$ having distribution 
$\Pi$ on a given set of vectors $\mathcal{X}$. Using rescaling, we can suppose that $ \left\Vert W\right\Vert_{2}\leq 1$ almost surely.
%  $\mathcal{X} = \left \{e_j,\; 1\leq j\leq p \right \}$,
%  where $e_j$ are the canonical basis vectors in $\bR^{p}$. 
% 
%Let $\Omega=\bE\left (W\,W^{T}\right )$ and $\sigma_{\max}(\Omega)$, $\sigma_{\min}(\Omega)$ denote respectively its maximal and minimal singular values.
Let $\bE \left (W\,W^{T}\right )= \Omega$ and  $\omega_{\max}$, $\omega_{\min}$ denote respectively its maximal and minimal singular values. We need the following assumption on the distribution of $W$.
 \begin{Assumption}\label{ass_om1}
The matrix  $\Omega = \bE \left (W\,W^{T}\right )$  is positive definite. % with  $ \left\Vert W\right\Vert_{2}\leq 1$.}
% The matrix $\Omega$ is positive definite. 
% $$\bE \left (W\,W^{T}\vert \,t\right )\geq \Omega.$$
 \end{Assumption}
%When $t$ and $W$ are independent, Assumption \ref{ass_om1} is satisfied if $\bE\left (W\,W^{T}\right )>0$. 
Let $\Vert A\Vert _{L_2(\Pi\otimes\mu)}^{2}=\bE\left(\langle X,A\rangle^{2}\right)$.  An easy computation leads to
 \begin{equation*} %\label{RI_exp_1}
 \begin{split}
 \Vert A\Vert _{L_2(\Pi\otimes\mu)}^{2}&
=\bE\left(\langle W,A\,\phi(t)\rangle^{2}\right)\\& 
 = \bE_{t}\left (\bE_{W}\left(\langle W,A\,\phi(t)\rangle^{2} \right)\right )
 \end{split}
 \end{equation*}
 and %, using Assumption \ref{ass_om1}, we obtain
  \begin{equation*}
  \begin{split}
  \bE_{W}\left(\langle W,A\,\phi(t)\rangle^{2} \right)&=\bE_{W}\left(\mathrm{tr}\left (\left (A\,\phi(t)\right )^{T} W\,W^{T}A\,\phi(t)\right ) \right)\\&=\bE_{W}\left(\mathrm{tr}\left ( W\,W^{T}A\,\phi(t)\left (A\,\phi(t)\right )^{T}\right ) \right)\\&= \left\langle \bE_{W}\left(W^{T}\,W\right),A\,\phi(t)\left (A\,\phi(t)\right )^{T}\right\rangle\\&= \left\langle \Omega,A\,\phi(t)\left (A\,\phi(t)\right )^{T}\right\rangle.
  \end{split}
  \end{equation*}

  By definition we obtain
   \begin{equation*} %\label{RI_esp_2}
   \begin{split}
   \left\langle \Omega,A\,\phi(t)\left (A\,\phi(t)\right )^{T}\right\rangle\geq 
\omega_{\min}\left\Vert A\,\phi(t)\right\Vert^{2}_{2}.
   \end{split}
   \end{equation*}
   Finally we compute
   \begin{equation}\label{RI_exp_orth}
   \begin{split}
   \Vert A\Vert _{L_2(\Pi\otimes\mu)}^{2}&\geq \omega_{\min}\,\bE_{t}\left ( \left\Vert A\,\phi(t)\right\Vert^{2}_{2}\right ) %\\&
   =\omega_{\min}\left\Vert A\right\Vert^{2}_{2}
   \end{split}
   \end{equation}
where in the last display we used that $(\phi_i(\cdot))_{i=1,\dots,\infty}$ is an orthonormal basis in $L_2\left ((0,1),d\mu\right )$.

We consider the case of \textit{sub-exponential noise} which satisfies the following condition
 \begin{Assumption}\label{noise} There exist a constant $K>0$ such that
 $$\underset{i=1,\dots,n}{\max}\bE\exp\left (\vert\xi_i\vert/K\right )\leq  e.$$
 \end{Assumption}
\noindent  
For instance, if $\xi_i$ are i.i.d. standard Gaussian we can take $K=1$.

 % % % % % % % % % % % % % % % % % % % % % % % % % % % % % % % % % % % % % % % % % % % % % % % % % % % %
\section{Main Results}\label{main}
   
  Let
  \begin{equation*} %\label{stoch1}
   \Sigma_R=\dfrac{1}{n} \Sum \epsilon_i X_i\qquad\text{and}\qquad \Sigma=\dfrac{1}{n}\Sum \left (W^{T}_i\rho^{(l)}(t_i)+\sigma\,\xi_i\right )X_i
   \end{equation*}
   where $\{\epsilon_i\}_{i=1}^{n}$ is an i.i.d. Rademacher sequence. These stochastic terms play an important role in the choice of the regularization parameter $\lambda$. 
   
   We introduce the following notations:
   
 $$M=\mathrm{tr}(\Omega)\vee \left (l\,\omega_{\max}\right )\quad\text{and}\quad n^{**}=\dfrac{C\,c^{2}_{\phi}\,l\,\log(d)}{\omega^{2}_{\min}}\left [(M\,s)\vee 1\right ]. $$
 %where $C$ is a large enough numerical constant.
%   and
%   \cred{   $$n_1=\dfrac{64\,c^{2}_{\phi}\,\log(d)}{\log\left (6/5\right )\,\omega_{\min}}\quad\text{and}\quad n^{**}=\dfrac{C\,c^{2}_{\phi}\,l\,s\,M\,\log(d)}{\omega^{2}_{\min}} .$$ }
%    
      The following theorem gives a general upper bound  on the prediction error for the estimator $\hat A$ given by \eqref{estimator_orth}. Its proof is given in Appendix \ref{proof-thm2}. 
 \begin{Theorem}\label{thm2}
 Let $\lambda\geq 3 \left\Vert \Sigma\right\Vert$ and suppose that Assumption  \ref{ass_om1} holds.
   Then, with probability at least $1-2/d$,
   \begin{itemize}
   \item[(i)] 
   \begin{equation*} 
    \begin{split}
   \hskip - 0.75 cm \left \Vert \hat A-A_0\right \Vert_2^{2}\leq C\,\max\left \{\dfrac{s}{\omega^{2}_{\min}}\left (
    \lambda^{2} +\left\Vert A_{0}\right\Vert^{2}_*\,\dfrac{c^{2}_{\phi}\,l\,M\,\log(d)}{n}\right ),\dfrac{c_{\phi}\,\left\Vert A_0\right\Vert^{2}_*}{\omega_{\min}}\,\sqrt{\dfrac{\log(d)\,l}{n}}\right \}.
    \end{split}
    \end{equation*}
    \item [(ii)] If, in addition $n\geq n^{**} $, then 
    \begin{equation*} %\label{3}
    \begin{split}
    \left \Vert \hat A-A_0\right \Vert_2^{2}\leq \dfrac{C\,s\,\lambda^{2}}{\omega^{2}_{\min}}
    \end{split}
    \end{equation*}
   \end{itemize}

 where $d=l+p$.
 \end{Theorem}

 In order to obtain   upper bounds in Theorem \ref{thm2} in a closed form, it is necessary to obtain a suitable upper bound  for 
 $\left\Vert \Sigma\right\Vert$.
 The following lemma, proved in Section \ref{stochastic}, gives such bound.

 \begin{lemma}\label{delta}
Under Assumptions \ref{approximation_sup} - \ref{noise}, there exists a numerical constant $c^{*}$, that depends only on $K$, such that,
     for all $t>0$  with probability at least $1-2e^{-t}$
      \begin{equation}\label{norm_delta}
          \begin{split}
          \left\Vert \Sigma\right\Vert&\leq \left (\sigma\,c^{*}+\dfrac{2\,b\,\sqrt{s-1}}{l^{\gamma}}\right )\max \left \{\sqrt{\dfrac{M\left (t+\log(d)\right )}{n}},\right. \\&\left .\hskip 4 cm\dfrac{c_{\phi}\,\sqrt{l}\,\left (t+\log(d)\right )\left (\left [K\,\log\left (\dfrac{K\,c_{\phi}}{\omega_{\max}}\right )\right ]\vee 1\right )}{n}\right \}
          \end{split}
          \end{equation}
  where $d=p+l$. 
 \end{lemma}

 The optimal choice of the parameter $t$ in Lemma \ref{delta} is $t=\log(d)$. 
Larger $t$ leads to a slower rate of convergence and a smaller $t$ does not improve the rate but
 makes the concentration probability smaller. With this choice of $t$, the second  terms in the maximum in \eqref{norm_delta} 
is negligibly small  for $n\geq n^{*}$ where
 $$n^{*}=\dfrac{2\,c^{2}_{\phi}\,l\,\left (\left [K\,\log\left (\dfrac{K\,c_{\phi}}{\omega_{\max}}\right )\right ]\vee 1\right )^{2}\,\log(d)}{M}.$$
   In order to satisfy condition $\lambda \geq 3\left\Vert \Sigma\right\Vert$ in Theorem \ref{thm2} we can choose 
  \begin{equation}\label{lambda_1}
  \lambda=4.25 \left (c^{*}\sigma+\dfrac{2\,b\,\sqrt{s-1}}{l^{\gamma}}\right ) \sqrt{\dfrac{M\,\log(d)}{n}}.
  \end{equation}
  If $\xi_i$ are $N(0,1)$, then we can take $c^{*}=6.5$ (see Lemma 4 in \cite{klopp-variance}).

%Let $n^{*}=n_1\vee n_2$.
 With these choices of $\lambda$, we obtain the following theorem. 

  \begin{Theorem}\label{thm3}
   Let Assumptions \ref{approximation_sup} - \ref{noise} hold. Consider regularization parameters $\lambda$ 
satisfying \eqref{lambda_1} and $n\geq n^{*}$. Then, with probability greater than $1-4/d$
  \begin{itemize}
  \item[(i)]\begin{equation*}
    \hskip - 1 cm \left \Vert \hat A-A_0\right \Vert_2^{2}\leq C\,\max\left \{\left (\sigma^{2}+\dfrac{b^{2}\,(s-1)}{l^{2\gamma}}+l\,\left \Vert A_0\right \Vert^{2}_*\right  )\dfrac{M\,s\,\log(d)}{n\,\omega^{2}_{\min}},\dfrac{c_{\phi}\,\left\Vert A_0\right\Vert^{2}_*}{\omega_{\min}}\,\sqrt{\dfrac{\log(d)\,l}{n}}\right \}.
    \end{equation*}
    \item[(ii)] If, in addition $n\geq n^{**}$, then
    \begin{equation*}
        \left \Vert \hat A-A_0\right \Vert_2^{2}\leq C \,\left (\sigma^{2}+\dfrac{b^{2}\,(s-1)}{l^{2\gamma}}\right  )\dfrac{M\,s\,\log(d)}{n\,\omega^{2}_{\min}}.
        \end{equation*}
  \end{itemize} 
  \end{Theorem}

  Using $\hat A$ we define the estimator of  $f(t)$ as 
  \begin{equation}\label{est_f}
  \hat f(t)=\left (\hat f_1(t),\dots,\hat f_p(t)\right )^{T}=\hat A\,\phi(t).
  \end{equation}

Theorem \ref{thm3} allows to obtain the following upper bounds on the prediction error of  $\hat f(t)$.

 \begin{Corollary}\label{cor_1}
Suppose that  the assumptions of Theorem \ref{thm3} hold.
With probability greater than $1-4/d$,  one has   
\begin{itemize}

\item [(a)] $\forall t\in\supp( \mu)$
 \begin{equation*}
 \begin{split}
\hskip - 1cm \dfrac{1}{p}\underset{i=1}{\overset{p}{\Sigma}}\vert \hat f_i(t)-f_i(t) \vert&\leq \dfrac{C\,
\left\Vert \phi(t)\right\Vert^{2}_{2}\,\beta}{n}
%\,\max\left \{\dfrac{\mathbf{\beta}\,M\,s\,\log(d)}{p\,n\,\omega^{2}_{\min}},\dfrac{c_{\phi}\,\left\Vert A_0\right\Vert^{2}_*}{\omega_{\min}\,p}\,\sqrt{\dfrac{\log(d)}{n}}\right \}
+\dfrac{2\,b^{2}\,s}{p\,l^{2\gamma}},
 \end{split}
 \end{equation*}
 \item [(b)] If, in addition, Assumption \ref{approximation_L2} holds
 \begin{equation*}
   \begin{split}
    \hskip - 1cm\dfrac{1}{p}\underset{i=1}{\overset{p}{\Sigma}}\Vert \hat f_i-f_i \Vert_{L_2(d\mu)}^{2}&\leq\dfrac{ C\,\beta}{n}
%    \max\left \{\dfrac{\mathbf{\beta}\,M\,s\,\log(d)}{p\,n\,\omega^{2}_{\min}},\dfrac{c_{\phi}\,\left\Vert A_0\right\Vert^{2}_*}{\omega_{\min}\,p}\,\sqrt{\dfrac{\log(d)}{n}}\right \}
   +\dfrac{2\,b_1^{2}\,s}{p\,l^{(2\gamma+1)}}, \\
 \end{split}
   \end{equation*}
\end{itemize}
   where 
    \begin{equation*}
         \beta = \left\{
    \begin{array}{ll}
                 \left (\sigma^{2}+\dfrac{b^{2}\,(s-1)}{l^{2\gamma}}\right )\,\dfrac{M\,s\,\log(d)}{p\,\omega^{2}_{\min}}, 
   & \hskip -1.5 cm\mbox{if} \quad  n\geq n^{**}\\  \\
   \max\left \{\left (\sigma^{2}+\dfrac{b^{2}\,(s-1)}{l^{2\gamma}}+l\,\left \Vert A_0\right \Vert^{2}_*\right )\,\dfrac{M\,s\,\log(d)}{p\,\omega^{2}_{\min}},\dfrac{c_{\phi}\,\left\Vert A_0\right\Vert^{2}_*\,\sqrt{\log(d)\,l\,n}}{\omega_{\min}\,p}\right \}, 
      &\hskip - 0.5 cm \mbox{if not} .
        \end{array} \right.
        \end{equation*}
%   $\mathbf{\beta}= \sigma^{2}+\dfrac{b^{2}\,s}{l^{2\gamma}}$ if $n\geq n^{**}$ and $\mathbf{\beta}= \sigma^{2}+\dfrac{b^{2}\,s}{l^{2\gamma}}+l\,\left \Vert A_0\right \Vert^{2}_*$ if not.
 \end{Corollary}

 \begin{proof} 
We shall prove the second statement of the corollary, the first one can be proved in a similar way.
 Let $A^{i}$ denote the $i$-th row of a matrix $A$. We compute
  \begin{equation}\label{corol_f_1}
  \begin{split}
   \left\Vert f_i(t)-\hat A^{i}\phi(t)\right\Vert_{L_2(d\mu)}&\leq \left\Vert f_i(t)-A_0^{i}\phi(t)\right\Vert_{L_2(d\mu)}+\left\Vert \left (A_0^{i}-\hat A^{i}\right )\phi(t)\right\Vert_{L_2(d\mu)}
   \\&=\left\Vert \rho^{(l)}_i(t)\right\Vert_{L_2(d\mu)}+\left\Vert A_0^{i}-\hat A^{i}\right\Vert_2
  \end{split}
  \end{equation}
  where in the last display we used that $(\phi_i(\cdot))_{i=1,\dots,\infty}$ is an orthonormal basis. Using \eqref{corol_f_1} and Assumption \ref{approximation_L2} we derive
  \begin{equation*} %\label{corol_f_1}
    \begin{split}
     \underset{i=1}{\overset{p}{\Sigma}}\Vert \hat f_i-f_i \Vert_{L_2(d\mu)}^{2}&\leq 
  \dfrac{2\,b_1^{2}\,s}{l^{(2\gamma+1)}}+2\,\left \Vert \hat A-A_0\right \Vert_2^{2}.
    \end{split}
    \end{equation*}
    Now Theorem \ref{thm3} implies the statement of the corollary.
 \end{proof}

 % % % % % % % % % % % % % % % % % % % % % % % % % % % % % % % % % % % % % % % % % % % % % % % % % % % % % % % % % % % %

  \section{Orthonormal dictionary}\label{orth}

 As an important particular case, let us consider the orthonormal dictionary. Let $(e_j)_{j}$ be the canonical basis of $\bR^{p}$.
Assume that the vectors $W_i$ are i.i.d copies of a random vector $W$ which has the uniform distribution $\Pi$ on the set 
  $$\mathcal{X} = \left \{e_j,\; 1\leq j\leq p \right \}.$$
   Note that this is an unfavorable   case of very ``sparse observations'',
  that is, each observation provides some information on only one of the coefficients of $f(t)$.

In this case,  $\Omega=\dfrac{1}{p}\mathbb{I}_p$, $\omega_{\max}=\omega_{\min}=\dfrac{1}{p}$ and we obtain
the following values of parameters
   \begin{equation}\label{lambda_orth}
    \begin{split}
    M&=\dfrac{l\vee p}{p},\\
    n^{*}&=2\,K^{2}\,\log^{2}\left (K\,p\right )\,c^{2}_{\phi}\,\log(d)\,(l\wedge p) ,\\
    \lambda&=4.25 \left (C^{*}\sigma+\dfrac{2\,b\,\sqrt{s-1}}{l^{\gamma}}\right ) \sqrt{\dfrac{(l\vee p)\,\log(d)}{p\,n}},\\ 
   n^{**}&=C\,c^{2}_{\phi}\,l\,s\,p\,(l\vee p)\,\log(d).
    \end{split}
    \end{equation}
    
 Plugging these values into Corollary \ref{cor_1},  we derive the following result.

 \begin{Corollary}\label{cor_orth_1}
  Let Assumptions \ref{approximation_sup} and  \ref{noise} hold. 
Consider  regularization parameter $\lambda$ satisfying \eqref{lambda_orth}, and $n\geq n^{*}$. 
Then, with probability greater than $1-4/d$, one has
\begin{itemize}
\item [(a)] $\forall t\in\supp( \mu)$
 \begin{equation}\label{upp_sup}
 \begin{split}
\hskip - 1cm \dfrac{1}{p}\underset{i=1}{\overset{p}{\Sigma}}\vert \hat f_i(t)-f_i(t) \vert&\leq \dfrac{C\,
\left\Vert \phi(t)\right\Vert^{2}_{2}\,\beta}{n}
%\,\max\left \{\dfrac{\mathbf{\beta}\,M\,s\,\log(d)}{p\,n\,\omega^{2}_{\min}},\dfrac{c_{\phi}\,\left\Vert A_0\right\Vert^{2}_*}{\omega_{\min}\,p}\,\sqrt{\dfrac{\log(d)}{n}}\right \}
+\dfrac{2\,b^{2}\,s}{p\,l^{2\gamma}},
 \end{split}
 \end{equation}
 \item [(b)] If, in addition, Assumption \ref{approximation_L2} holds
 \begin{equation}\label{upp_L2}
   \begin{split}
    \hskip - 1cm\dfrac{1}{p}\underset{i=1}{\overset{p}{\Sigma}}\Vert \hat f_i-f_i \Vert_{L_2(d\mu)}^{2}&\leq\dfrac{ C\,\beta}{n}
%    \max\left \{\dfrac{\mathbf{\beta}\,M\,s\,\log(d)}{p\,n\,\omega^{2}_{\min}},\dfrac{c_{\phi}\,\left\Vert A_0\right\Vert^{2}_*}{\omega_{\min}\,p}\,\sqrt{\dfrac{\log(d)}{n}}\right \}
   +\dfrac{2\,b_1^{2}\,s}{p\,l^{(2\gamma+1)}}, \\
 \end{split}
   \end{equation}
\end{itemize}
   where 
    \begin{equation*}
        \beta = \left\{
    \begin{array}{ll}
                 \left (\sigma^{2}+\dfrac{b^{2}\,(s-1)}{l^{2\gamma}}\right )\,(l\vee p)\,s\,\log(d), 
   & \hskip - 1.5 cm \mbox{if} \quad  n\geq n^{**}\\  \\
   \left (\sigma^{2}+\dfrac{b^{2}\,(s-1)}{l^{2\gamma}}+l\,\left \Vert A_0\right \Vert^{2}_*\right )\,(l\vee p)\,s\,\log(d), 
      & \mbox{\text{if not}} .
        \end{array} \right.
        \end{equation*}

%  \item [(a)]  $\forall t\in\supp( \mu)$
%   \begin{equation}  \label{eq:point_risk}
%   \begin{split}
%    \dfrac{1}{p}\underset{i=1}{\overset{p}{\Sigma}}\vert \hat f_i(t)-f_i(t) \vert&  
%  \leq C\,\max\left \{\left\Vert \phi(t)\right\Vert^{2}_{2}\,\dfrac{(l\vee p)\,s\,\log(d)}{n},\dfrac{c_{\phi}\,\left\Vert A_0\right\Vert^{2}_*}{\omega_{\min}}\,\sqrt{\dfrac{\log(d)}{n}}\right \}
%        +\dfrac{2\,b^{2}\,s}{p\,l^{2\gamma}},
%   \end{split}
%   \end{equation}
%   \item [(b)] if, in addition, Assumption \ref{approximation_L2} holds
%    \begin{equation}  \label{eq:l2risk}
%     \begin{split}
%      \dfrac{1}{p}\underset{i=1}{\overset{p}{\Sigma}}\Vert \hat f_i-f_i \Vert_{L_2(d\mu)}^{2}
%    & \leq\mathbf{C}\,\dfrac{(l\vee p)\,s\,\log(d)}{n}+\dfrac{2\,b^{2}\,s}{p\,l^{(2\gamma+1)}}.
%     \end{split}
%     \end{equation}
%\end{itemize}
% where 
%   $\mathbf{C}= C \left (\sigma^{2}+b^{2}\right )$ if $n\geq n^{**}$ and 
%$\mathbf{C}= C \left (\sigma^{2}+b^{2}+l\,\left \Vert A_0\right \Vert^{2}_*\right )$ if not. 
\end{Corollary}
\textbf{Remarks.} \textit{Optimal choice of parameter $l$:}  The upper bounds given in Corollary \ref{cor_orth_1} indicate the optimal choice of parameter $l$. From  \eqref{lambda_orth} we compute the following values of $l$:
 \begin{equation*}
 \begin{split}
 l_1^{*}=\dfrac{n}{C\,c^{2}_{\phi}\,s\,p^{2}\,\log(d)}\quad\text{if}\quad l\leq p
 \end{split}
 \end{equation*}
 and
 \begin{equation*}
  \begin{split}
  l_2^{*}=\sqrt{\dfrac{n}{C\,c^{2}_{\phi}\,s\,p\,\log(d)}}\quad\text{if}\quad l> p.
  \end{split}
  \end{equation*}
  Let  \begin{equation*}
   \begin{split}
   F_1(l)&=C\,\left (\sigma^{2}+\dfrac{b^{2}\,(s-1)}{l^{2\gamma}}\right )\,\dfrac{p\,s\,\log(d)}{n}+\dfrac{2\,b_1^{2}\,s}{p\,l^{(2\gamma+1)}}, \\
   F_2(l)&=F_1(l)+l\,\left \Vert A_0\right \Vert^{2}_*\dfrac{p\,s\,\log(d)}{n},\\
   F_3(l)&=C\,\left (\sigma^{2}+\dfrac{b^{2}\,(s-1)}{l^{2\gamma}}\right )\,\dfrac{l\,s\,\log(d)}{n}+\dfrac{2\,b_1^{2}\,s}{p\,l^{(2\gamma+1)}}, \\
      F_4(l)&=F_3(l)+l^{2}\,\left \Vert A_0\right \Vert^{2}_*\dfrac{s\,\log(d)}{n}.
   \end{split}
   \end{equation*}
%We assume that $p$ is large
Let $\gamma\geq 1/2$ and consider first the case  $ s\,p^{3}\,\log(d) \gtrsim n \gtrsim s\,p^{2}\,\log(d)$ (the symbol $\lesssim$ means that the inequality holds up to a multiplicative numerical constant). Then, Corollary \ref{cor_orth_1} implies that
 \begin{equation*}
\dfrac{1}{p}\underset{i=1}{\overset{p}{\Sigma}}\Vert \hat f_i-f_i \Vert_{L_2(d\mu)}^{2}\leq \left\{
     \begin{array}{lll}
                  F_1(l), 
    & \mbox{if} \quad  1\leq l\leq l_1^{*} \\
    F_2(l), 
       & \mbox{\text{if}} \quad  l_1^{*}< l\leq p\\
        F_4(l), 
              & \mbox{\text{if}} \quad   l> p.
         \end{array} \right.
 \end{equation*}
On $[1,l_1^{*}],$ $F_1(l)$ achieves its minimum at $l_1^{*}$. Note that $F_1(l_1^{*})\leq F_2(l)$ for any $l\in [l_1^{*},p]$ and $F_1(l_1^{*})\leq F_4(l)$ for any $l>p$.
Then, for $ s\,p^{3}\,\log(d) \gtrsim n \gtrsim s\,p^{2}\,\log(d)$ the optimal value of $l$ minimizing \eqref{upp_L2} is
$$\hat l_1=\left [\dfrac{n}{C\,c^{2}_{\phi}\,s\,p^{2}\,\,\log(d)}\right ].$$
When $   n \gtrsim s\,p^{3}\,\log(d)$, the Corollary \ref{cor_orth_1} implies that
 \begin{equation*}
\dfrac{1}{p}\underset{i=1}{\overset{p}{\Sigma}}\Vert \hat f_i-f_i \Vert_{L_2(d\mu)}^{2}\leq \left\{
     \begin{array}{lll}
                  F_1(l), 
    & \mbox{if} \quad  1\leq l\leq p \\
    F_3(l), 
       & \mbox{\text{if}} \quad  p< l\leq l_2^{*}\\
        F_4(l), 
              & \mbox{\text{if}} \quad   l>l_2^{*}.
         \end{array} \right.
 \end{equation*}
 Let $$l_3^*=\left (\dfrac{C\,n}{\sigma^{2}\,p\,\log(d)}\right )^{\dfrac{1}{2\gamma+2}}.$$
On $[p,l_2^{*}],$ $F_3(l)$ achieves its minimum at $l_2^{*}$ if $  p^{3+2\gamma}\,\log(d)\gtrsim n \gtrsim s\,p^{3}\,\log(d)$ and at $l_3^{*}$ if $   n \gtrsim p^{3+2\gamma}\,\log(d) $. Note that $F_3(l_2^{*})\leq F_1(l)$ for any $l\in [1,p]$ and $F_3(l_2^{*})\leq F_4(l)$ for any $l>l_2^{*}$.
Then, for $  p^{3+2\gamma}\,\log(d)\gtrsim n \gtrsim s\,p^{3}\,\log(d)$ the optimal value of $l$ minimizing \eqref{upp_L2} is
$$\hat l_2=\left [\sqrt{\dfrac{n}{C\,c^{2}_{\phi}\,s\,p\,\,\log(d)}}\right ]$$
and for $   n \gtrsim p^{3+2\gamma}\,\log(d) $ the optimal value of $l$ is
$$\hat l_3=\left (\dfrac{C\,n}{\sigma^{2}\,p\,\log(d)}\right )^{\dfrac{1}{2\gamma+2}}.$$

\textit{Minimax rate of convergence:} For $p=1$ the optimal choice of $l$ in \eqref{upp_L2} is 
$$\hat l= \left (\dfrac{2\,(2\,\gamma+1)\,b^{2}\,n}{\sigma^{2}\,\log(d)}\right )^{\dfrac{1}{2\gamma+2}}.$$
With this choice of $l$, the rate of convergence given by Corollary \ref{cor_orth_1} is $n^{-\dfrac{2\gamma+1}{2\gamma+2}}$. Note that for  $f\in \bW_{\gamma }(0,1)$
we recover the minimax rate of convergence  as  given in e.g. \cite{tsybakov_book}.

\section*{Acknowledgements}

Marianna Pensky was supported in part by National Science Foundation
(NSF), grant  DMS-1106564. The authors want to thank Alexander Tsybakov for 
extremely valuable discussions and suggestions.

%%%%%%%%%%%%%%%%%%%%%%%%%%%%%%%%%%%%%%%%%%%%%%%%%%%%%%%%%%%%%%%%%%%%%%%%%%%%%%%%%%%%%%%%%%%%%%%%%%%%%%%%%%%%%%%%%%%%%%%%%%

\vspace{10mm}

 \appendix

\noindent
{{\LARGE{\bf Appendix}}

\section{Proof of Theorem \ref{thm2}}\label{proof-thm2}

This proof uses ideas developed in the proof of Theorem 3 in \cite{klopp-general}. 
The main difference is that here we have no restriction on the $\sup-$norm of $A_0$. 
This implies several modifications in the proof.

It follows from the definition of the estimator $\hat A$ that 
\begin{equation*}
\begin{split}
\dfrac{1}{n}\Sum \left (Y_i-\left\langle X_i,\hat A\right\rangle\right )^{2}+\lambda \Vert \hat A\Vert_*&\leq \dfrac{1}{n}\Sum \left (Y_i-\left\langle X_i,A_0\right\rangle\right )^{2}+\lambda \Vert A_0\Vert_*
\end{split}
\end{equation*}
which, due to %using
 \eqref{model}, implies 
\begin{equation}\label{new_1}
\begin{split}
\dfrac{1}{n}\Sum \left (\left\langle X_i,A_0-\hat A\right\rangle+W^{T}_i\rho^{(l)}(t_i)+\xi_i\right )^{2}+\lambda \Vert \hat A\Vert_*&\leq \\&\hskip -2 cm\dfrac{1}{n}\Sum \left (W^{T}_i\rho^{(l)}(t_i)+\xi_i\right )^{2}+\lambda \Vert A_0\Vert_*.
\end{split}\end{equation}
Set $H=A_0-\hat A$ and $\Sigma=\dfrac{1}{n}\Sum \left (W^{T}_i\rho^{(l)}(t_i)+\xi_i\right )X_i$. Then, we can write \eqref{new_1} in the following way
\begin{equation*}
\dfrac{1}{n}\Sum \left\langle X_i,H\right\rangle^{2}+2\left\langle \Sigma,H\right\rangle\\+\lambda \Vert \hat A\Vert_*\leq \lambda \Vert A_0\Vert_*.
\end{equation*}
By duality between the nuclear  and the operator norms, we obtain
\begin{equation}\label{1}
\dfrac{1}{n}\Sum \left\langle X_i,H\right\rangle^{2}+ \lambda\Vert \hat A\Vert_*\leq 2\left\Vert \Sigma\right\Vert \Vert H\Vert_*+\lambda \Vert A_0\Vert_*.
\end{equation}

Let $P_S$ denote the projector on the linear subspace $S$ and let $S^\bot$ be the orthogonal complement of $S$.
Let $u_j(A)$ and $v_j(A)$ denote respectively the \textit{left} and the \textit{right} orthonormal \textit{singular vectors} of $A$, $S_1(A)$ is the linear span of $\{u_j(A)\}$, $S_2(A)$ is the linear span of $\{v_j(A)\}$. For $A,B\in \mathbb{R}^{p\times l}$ we set
 $\mathbf P_A^{\bot}(B)=P_{S_1^{\bot}(A)}BP_{S_2^{\bot}(A)}$ and $\mathbf P_A(B)=B- \mathbf P_A^{\bot}(B)$.
 
 By definition, for any matrix $B$, the singular vectors of $\mathbf P_{A_0}^{\bot}(B)$  are orthogonal to the space spanned by the singular vectors of $A_0$. This implies that $\left\Vert A_0+\mathbf P_{A_0}^{\bot}(H) \right\Vert_1=\left \Vert A_0 \right \Vert_*+\left\Vert \mathbf P_{A_0}^{\bot}(H) \right\Vert_*$. Then we compute
 \begin{equation}\label{ineq}
 \begin{split}
 \left\Vert \hat A\right\Vert_*&=\Big\Vert A_0 +H\Big\Vert_*\\
 &=\left\Vert A_0 +\mathbf P_{A_0}^{\bot}(H)+\mathbf P_{A_0}(H)\right\Vert_*\\
 &\geq \left\Vert A_0 +\mathbf P_{A_0}^{\bot}(H)\right\Vert_*-\left\Vert\mathbf P_{A_0}(H)\right\Vert_*\\
 &=\left \Vert A_0 \right \Vert_*+\left\Vert \mathbf P_{A_0}^{\bot}(H)\right\Vert_*-\left\Vert\mathbf P_{A_0}(H)\right\Vert_*.
 \end{split}
 \end{equation}
 From \eqref{ineq} we obtain
\begin{equation}\label{un2}
\left \Vert A_0 \right \Vert_*-\left\Vert \hat A\right\Vert_*\leq \left\Vert\mathbf P_{A_0}(H)\right\Vert_*-\left\Vert \mathbf P_{A_0}^{\bot}(H)\right\Vert_*.
\end{equation}
From \eqref{1}, using \eqref{un2} and $\lambda\geq 3\left \Vert \Sigma\right\Vert$ we obtain
\begin{equation}\label{2}
\begin{split}
\dfrac{1}{n}\Sum \left\langle X_i,H\right\rangle^{2}&\leq 2\left \Vert \Sigma\right \Vert \left \Vert \mathbf P_{A_0}\left (H\right )\right \Vert_*+\lambda\left \Vert\mathbf P_{A_0}\left ( H\right )\right \Vert_*\\&\leq \dfrac{5}{3}\lambda  \left \Vert \mathbf P_{A_0}\left (H\right )\right \Vert_*.
\end{split}
\end{equation}
 Since $\mathbf P_A(B)=P_{S_1^{\bot}(A)}BP_{S_2(A)}+ P_{S_1(A)}B$
and $\rank\,(P_{S_i(A)}B)\leq \rank \,(A)$ we derive  that $\rank\,(\mathbf P_A(B))\leq 2\,\rank\,(A)$. 
From \eqref{2} we compute
\begin{equation}\label{3}
\dfrac{1}{n}\Sum \left\langle X_i,H\right\rangle^{2}\leq \dfrac{5}{3}\,\lambda\,\sqrt{2\,R}\left\Vert H \right\Vert_2
\end{equation}
where we set $R=\rank\,(A_0)$.

%Let $ \Vert A\Vert _{L_2(\Pi\otimes\mu)}^{2}=\bE\left(\langle A,X\rangle^{2}\right)$.  In the case of the orthonormal dictionary we have that
%  \begin{equation}\label{RI_exp_orth}
%  \begin{split}
%   \Vert A\Vert _{L_2(\Pi\otimes\mu)}^{2}=\dfrac{\Vert A\Vert _2^{2}}{p}.
%  \end{split}
%  \end{equation}
For $0<r\leq m=\min\,(p,l)$ we consider the following constraint set
\begin{equation}\label{constrain}
\mathcal{C}(r)=\left \{\left\Vert A\right\Vert_2 \leq 1,\,  \left\Vert A\right\Vert_{L_2(\Pi\otimes\mu)}^{2}\geq c_{\phi}\,\sqrt{\dfrac{64\,\log(d)\,l}{\log\left (6/5\right )\,n}}, \left\Vert A\right\Vert_{*}\leq \sqrt{r} \left\Vert A\right\Vert_{2}\right \}
\end{equation}
where $\Vert A\Vert _{L_2(\Pi\otimes\mu)}^{2}=\bE\left(\langle X,A\rangle^{2}\right)$. Note that the condition $\left\Vert A\right\Vert_{*}\leq \sqrt{r} \left\Vert A\right\Vert_{2}$ is satisfied if $\rank(A)\leq r$.

The following lemma shows that for matrices $A\in \mathcal{C}(r)$ we have some approximative restricted isometry. Its proof is given in Appendix \ref{proof-thm1}.
\begin{lemma}\label{thm1}
For all $A\in \mathcal{C}(r)$
$$\dfrac{1}{n}\Sum \left\langle X_i,A\right\rangle^{2}\geq \dfrac{1}{2}\Vert A\Vert _{L_2(\Pi\otimes\mu)}^{2}-\dfrac{44\,c^{2}_{\phi}\,l\,r}{\omega_{\min}}\left (\bE\left ( \left\Vert \Sigma_R\right\Vert\right )\right )^{2} $$
with probability at least $1-\dfrac{2}{d}$.
\end{lemma}

 We need the following auxiliary lemma which is proved in Appendix \ref{pl2}.
\begin{lemma}\label{l2} If $\lambda_1>3\left\Vert \Sigma\right\Vert$ 
$$ \left\Vert \mathbf P_{A_0}^{\bot}(H) \right\Vert_*\leq 5\left\Vert \mathbf P_{A_0}(H) \right\Vert_*
.$$
\end{lemma}
 Lemma \ref{l2} implies that 
  \begin{equation}\label{new_2}
  \begin{split}
  \left\Vert H\right\Vert_{*}&\leq 6\left\Vert \mathbf P_{A_0}(H) \right\Vert_*
  \\&\leq \sqrt{72\,R} \left\Vert H\right\Vert_{2}.
  \end{split}
  \end{equation}

%Note that condition $n\geq n_1$ implies that $$\left\Vert H\right\Vert_{L_2(\Pi\otimes\mu)}^{2} \geq \left\Vert H\right\Vert^{2}_2\,\dfrac{64\,c^{2}_{\phi}\,\log(d)}{\log\left (6/5\right )\,n}.$$ 
%If not, using \eqref{RI_exp_orth}, we get
%\begin{equation*} %\label{o1}
%\omega_{\min}\,\left\Vert H\right\Vert^{2}_2< \left \Vert H\right\Vert^{2}_2\,\dfrac{64\,c^{2}_{\phi}\,\log(d)}{\log\left (6/5\right )\,n}
%\end{equation*}
% and $n<n_1$.

 If $\left\Vert H\right\Vert_{L_2(\Pi\otimes\mu)}^{2}\geq c_{\phi}\,\left \Vert H\right\Vert^{2}_2\,\sqrt{\dfrac{64\,\log(d)\,l}{\log\left (6/5\right )\,n}}$,  \eqref{new_2} implies that $ \dfrac{H}{\left\Vert H \right\Vert_{2}}\in \mathcal{C}\left (72\,R\right )$ 
  and we can apply Lemma \ref{thm1}. From Lemma \ref{thm1} and \eqref{3} we obtain that 
with probability at least $1-\dfrac{2}{d}$ one has
\begin{equation} \label{trace-3}
\begin{split}
\dfrac{1}{2}\Vert H\Vert _{L_2(\Pi\otimes\mu)}^{2}&\leq  \dfrac{5}{3}\lambda\sqrt{2\,R}\left\Vert H \right\Vert_2+\dfrac{3168\,c^{2}_{\phi}\,l\,R}{\omega_{\min}}\left\Vert H\right\Vert^{2}_2\,\left (\bE\left ( \left\Vert \Sigma_R\right\Vert\right )\right )^{2}.
\end{split}
\end{equation}
The following Lemma, proved in Section \ref{proof_E}, gives a suitable bound on $\bE \left\Vert \Sigma_R\right\Vert$:
\begin 
   {lemma}\label{Edelta}
   Let $(\epsilon_i)_{i=1}^{n}$ be an i.i.d. Rademacher sequence. Suppose that Assumption  \ref{ass_om1} holds. Then,
    $$ \bE \left\Vert \Sigma_R\right\Vert\leq 4.6 \sqrt{\dfrac{M\,\log(d)}{n}}$$ 
    where $d=p+l$ and $M=\mathrm{tr}(\Omega)\vee \left (l\omega_{\max}\right )$.
   \end{lemma} 
Using Lemma \ref{Edelta}, \eqref{RI_exp_orth} and \eqref{trace-3} we obtain
\begin{equation} \label{trace-4}
\begin{split}
\omega_{\min}\Vert H\Vert _2^{2}&\leq  \dfrac{10}{3}\lambda\sqrt{2\,R}\left\Vert H \right\Vert_2+\dfrac{C\,c^{2}_{\phi}\,l\,R\,M\,\log(d)}{\omega_{\min}\,n}\left\Vert H\right\Vert^{2}_2.
\end{split}
\end{equation}
On the other hand, equation \eqref{1} and the triangle inequality imply that
\begin{equation*} 
 \lambda\Vert \hat A\Vert_*\leq 2\left\Vert \Sigma\right\Vert \Vert \hat A\Vert_*+2\left\Vert \Sigma\right\Vert \Vert  
A_0\Vert_*+\lambda \Vert A_0\Vert_*
\end{equation*}
and $\lambda \geq  3\,\left\Vert \Sigma\right\Vert$ gets
\begin{equation} \label{bound_hat_zero}
\Vert \hat A\Vert_2\leq \Vert \hat A\Vert_*\leq 5\Vert A_0\Vert_*.
\end{equation}

Putting \eqref{bound_hat_zero} into \eqref{trace-4} and using $\rank(A_0)\leq s$ we compute
\begin{equation*} %\label{3}
\begin{split}
\left\Vert H\right\Vert^{2}_2\leq \dfrac{C\,s}{\omega^{2}_{\min}}\left (
 \lambda^{2} +\dfrac{c^{2}_{\phi}\,l\,M\,\log(d)\,\left\Vert A_{0}\right\Vert^{2}_*}{n}\right )
\end{split}
\end{equation*}
which implies the statement (i) of Theorem \ref{thm2} in the case when $\left\Vert H\right\Vert_{L_2(\Pi\otimes\mu)}^{2}\geq c_{\phi}\,\left \Vert H\right\Vert^{2}_2\,\sqrt{\dfrac{64\,\log(d)\,l}{\log\left (6/5\right )\,n}}$.

If $\left\Vert H\right\Vert_{L_2(\Pi\otimes\mu)}^{2}\leq c_{\phi}\,\left \Vert H\right\Vert^{2}_2\,\sqrt{\dfrac{64\,\log(d)\,l}{\log\left (6/5\right )\,n}}$, using \eqref{RI_exp_orth}, 
we derive
\begin{equation} \label{o1}
\omega_{\min}\,\left\Vert H\right\Vert^{2}_2\leq c_{\phi}\left \Vert H\right\Vert^{2}_2\,\sqrt{\dfrac{64\,\log(d)\,l}{\log\left (6/5\right )\,n}}.
\end{equation}
Then \eqref{bound_hat_zero} implies
 \begin{equation*}
 \begin{split}
 \left\Vert H\right\Vert^{2}_2&<  \dfrac{C\,c_{\phi}\,\left\Vert A_0\right\Vert^{2}_*}{\omega_{\min}}\,\sqrt{\dfrac{\log(d)\,l}{n}}.
 \end{split}
 \end{equation*}
%where we used $\dfrac{M}{\omega_{\min}}\geq \dfrac{\mathrm{tr}(\Omega)}{\omega_{\min}} >1$. 
This completes the proof of part (i) of Theorem \ref{thm2}.

If, in addition $n> 2\,\dfrac{C\,c^{2}_{\phi}\,l\,s\,M\,\log(d)}{\omega^{2}_{\min}}$, from \eqref{trace-4} we obtain
\begin{equation*} %\label{trace-4}
\begin{split}
\omega_{\min}\Vert H\Vert _2^{2}&\leq  \dfrac{10}{3}\lambda\sqrt{2\,R}\left\Vert H \right\Vert_2+\dfrac{\omega_{\min}}{2}\left\Vert H\right\Vert^{2}_2
\end{split}
\end{equation*}
and
\begin{equation*} %\label{3}
\begin{split}
\left\Vert H\right\Vert^{2}_2\leq \dfrac{C\,s\,\lambda^{2}}{\omega^{2}_{\min}}.
\end{split}
\end{equation*}
On the other hand, for $n>n^{**}$ \eqref{o1} does not hold.
%\begin{equation*} %\label{o1}
%\omega_{\min}\,\left\Vert H\right\Vert^{2}_2> \left \Vert H\right\Vert^{2}_2\,c_{\phi}\,\sqrt{\dfrac{64\,\log(d)}{\log\left (6/5\right )\,n}}.
%\end{equation*}
This completes the proof of Theorem \ref{thm2}.

%%%%%%%%%%%%%%%%%%%%%%%%%%%%%%%%%%%%%%%%%%%%%%%%%%%%%%%%%%%%%%%%%%%%%%%%%%%%%%%%%%%%%%%%%%%%%%%%%%

\section{Proof of Lemma \ref{thm1}}\label{proof-thm1}

Set $\mathcal{E}=\dfrac{44\,c^{2}_{\phi}\,l\, r\left (\bE\left ( \left\Vert \Sigma_R\right\Vert\right )\right )^{2}}{\omega_{\min}}$. 
We will show that the probability of the following bad event is small
\begin{equation*}
\mathcal{B}=\left \{\exists\,A\in \mathcal{C}(r)\,\text{such that}\,\left \vert\dfrac{1}{n} \Sum \left\langle X_i,A\right\rangle^{2} -\Vert A\Vert _{L_2(\Pi\otimes\mu)}^{2}\right \vert> \dfrac{1}{2}\Vert A\Vert _{L_2(\Pi\otimes\mu)}^{2}+ \mathcal{E}\right \}.
\end{equation*}
Note that $\mathcal{B}$ contains the complement of the event that we are interested in.

In order to estimate the probability of $\mathcal{B}$ we use a standard peeling argument. Let $\nu=c_{\phi}\sqrt{\dfrac{64\,\log(d)\,l}{\log\left (6/5\right )\,n}}$ and $\alpha=\dfrac{6}{5}$. For $k\in\mathbb N$ set $$S_k=\left \{A\in \mathcal{C}(r)\,:\,\alpha^{k-1}\nu \leq \Vert A\Vert _{L_2(\Pi\otimes\mu)}^{2}\leq \alpha^{k}\nu\right \}.$$ 
If the event $\mathcal{B}$ holds for some matrix $A\in \mathcal{C}(r)$, then $A$ belongs to some $S_k$ and 
\begin{equation}\label{Bl}
\begin{split}
\left \vert\dfrac{1}{n} \Sum \left\langle X_i,A\right\rangle^{2}-\Vert A\Vert _{L_2(\Pi\otimes\mu)}^{2}\right \vert&> \dfrac{1}{2}\Vert A\Vert _{L_2(\Pi\otimes\mu)}^{2}+ \mathcal{E}\\&> \dfrac{1}{2}\alpha^{k-1}\nu+ \mathcal{E}\\&
= \dfrac{5}{12}\alpha^{k}\nu+ \mathcal{E}.
\end{split}
\end{equation}
For each $T>\nu$ consider the following set of matrices
$$\mathcal{C}(r,T)=\left \{A\in\mathcal{C}(r) \,:\,  \left\Vert A\right\Vert_{L_2(\Pi\otimes\mu)}^{2}\leq T \right \}
$$ 
and the following event $$\mathcal{B}_k=\left \{\exists\,A\in \mathcal{C}(r,\alpha^{k}\nu)\,:\,\left \vert\dfrac{1}{n} \Sum \left\langle X_i,A\right\rangle^{2}-\Vert A\Vert _{L_2(\Pi\otimes\mu)}^{2}\right \vert> \dfrac{5}{12}\alpha^{k}\nu+ \mathcal{E}\right \}.$$
Note that $A\in S_k$ implies that $A\in \mathcal{C}(r,\alpha^{k}\nu)$. Then \eqref{Bl} implies that $\mathcal{B}_k$ holds and we obtain $\mathcal{B}\subset\cup \,\mathcal{B}_k$. Thus, it is enough to estimate the probability of the simpler event $\mathcal{B}_k$ and then to apply the union bound. Such an estimation is given by the following lemma. Its proof is given in Appendix \ref{pl1}. Let
$$Z_T=\underset{A\in \mathcal{C}(r,T)}{\sup}\left \vert\dfrac{1}{n} \Sum \left\langle X_i,A\right\rangle^{2}-\Vert A\Vert _{L_2(\Pi\otimes\mu)}^{2}\right \vert.$$
\begin{lemma}\label{l1}
%There exists numerical constants $(c_2,c_3)$, such that
 $$\mathbb P\left (Z_T>\dfrac{5}{12}T+ \dfrac{44\,c^{2}_{\phi}\,l\, r}{\omega_{\min}}\left (\bE \left\Vert \Sigma_R\right\Vert\right )^{2}\right )\leq \exp\left (-\dfrac{c_3nT^{2}}{c^{2}_{\phi}\,l}\right )$$
 where $c_3=\dfrac{1}{128}$.
\end{lemma}
Lemma \ref{l1} implies that $\mathbb P\left (\mathcal{B}_k\right )\leq \exp\left (-\dfrac{c_3\,n\,\alpha^{2k}\,\nu^{2}}{c^{2}_{\phi}\,l}\right )$. Using the union bound we obtain
\begin{equation*}
\begin{split}
\mathbb P\left (\mathcal{B}\right )&\leq \underset{k=1}{\overset{\infty}{\Sigma}}\mathbb P\left (\mathcal{B}_k\right ) % \\&
\leq \underset{k=1}{\overset{\infty}{\Sigma}}\exp\left (-\dfrac{c_3\,n\,\alpha^{2k}\,\nu^{2}}{c^{2}_{\phi}\,l}\right )\\&\leq \underset{k=1}{\overset{\infty}{\Sigma}}\exp\left (-\dfrac{\left (2\,c_3\,n\log(\alpha)\,\nu^{2}\right )k}{c^{2}_{\phi}\,l}\right )
\end{split}
\end{equation*} 
where we used $e^{x}\geq x$. We finally compute for $\nu= c_{\phi}\,\sqrt{\dfrac{64\,\log(d)\,l}{\log\left (6/5\right )\,n}}$
\begin{equation*}
\mathbb P\left (\mathcal{B}\right )\leq \dfrac{\exp\left (-\dfrac{2\,c_3\,n\log(\alpha)\,\nu^{2}}{c^{2}_{\phi}\,l}\right )}{1-\exp\left (-\dfrac{2\,c_3\,n\log(\alpha)\,\nu^{2}}{c^{2}_{\phi}\,l}\right )}=\dfrac{\exp\left (-\log(d)\right )}{1-\exp\left (-\log(d)\right )}.
\end{equation*} 
This completes the proof of Lemma \ref{thm1}.

%%%%%%%%%%%%%%%%%%%%%%%%%%%%%%%%%%%%%%%%%%%%%%%%%%%%%%%%%%%%%%%%%%%%%%%%%%%%%%%

 \section{Proof of Lemma \ref{l1}}\label{pl1}
 
Our approach is standard: first we show that $Z_T$ concentrates 
around its expectation and then we 
upper bound the expectation.
By definition, $$Z_T=\underset{A\in \mathcal{C}(r,T)}{\sup}\left \vert\dfrac{1}{n} \Sum \left\langle X_i,A\right\rangle ^{2}-\bE\left (\left\langle X,A\right\rangle ^{2}\right )\right \vert.$$
 Note that $$\left \vert \left\langle X_i,A\right\rangle\right \vert \leq  \left\Vert W\right\Vert_{2}\left\Vert \phi(t)\right\Vert_{2}\left\Vert A\right\Vert_{2}\leq c_{\phi}\,\sqrt{l},$$ where we used $\left\Vert W\right\Vert_{2}\leq 1$ and condition \eqref{eq:basis_assum}.
 
  Massart's concentration inequality (see e.g. \cite[Theorem 14.2]{sara}) implies that
\begin{equation}\label{concentration}
\mathbb P\left  (Z_T\geq \bE \left ( Z_T\right )+\dfrac{1}{9}\dfrac{5}{12}T\right )\leq \exp\left (-\dfrac{c_3nT^{2}}{c^{2}_{\phi}\,l}\right ).
\end{equation}
where $c_3=\dfrac{1}{128}$.

Next we bound the expectation $\bE\left ( Z_T\right )$. Using a standard symmetrization argument (see Ledoux and Talagrand \cite{ledoux_talagrand}) we obtain 
\begin{equation*}
\begin{split}
\bE \left ( Z_T\right )&= \bE\left (\underset{A\in \mathcal{C}(r,T)}{\sup}\left \vert\dfrac{1}{n} \Sum \left\langle X_i,A\right\rangle ^{2}-\bE\left (\left\langle X,A\right\rangle ^{2}\right )\right \vert\right )\\&\leq 2\bE\left (\underset{A\in \mathcal{C}(r,T)}{\sup}\left \vert\dfrac{1}{n} \Sum \epsilon_i\left\langle X_i,A\right\rangle ^{2} \right \vert\right )
\end{split}\end{equation*}
where $\{\epsilon_i\}_{i=1}^{n}$ is an i.i.d. Rademacher sequence.  Then, the contraction inequality (see Ledoux and Talagrand \cite{ledoux_talagrand}) yields
 \begin{equation*}
 \begin{split}
 \bE \left ( Z_T\right )&\leq 8\,c_{\phi}\,\sqrt{l}\,\bE\left (\underset{A\in \mathcal{C}(r,T)}{\sup}\left \vert\dfrac{1}{n} \Sum \epsilon_i\left\langle X_i,A\right\rangle\right \vert\right )\\&=8\,c_{\phi}\,\sqrt{l}\,\bE\left (\underset{A\in \mathcal{C}(r,T)}{\sup}\left \vert \left\langle \Sigma_R,A\right\rangle\right \vert\right )
 \end{split}
 \end{equation*}
where $\Sigma_R=\dfrac{1}{n} \Sum \epsilon_i X_i$. For $A\in \mathcal{C}(r,T)$ we have that 
\begin{equation*}
\begin{split}
\left\Vert A \right\Vert_* &\leq \sqrt{r}\left\Vert A\right\Vert_2\\&\leq  \dfrac{\sqrt{r}\left\Vert A\right\Vert_{L_2(\Pi\otimes\mu)}}{\sqrt{\omega_{\min}}}\\&\leq \sqrt{\dfrac{r\,T}{\omega_{\min}}}
\end{split}
\end{equation*}
where we have used \eqref{RI_exp_orth}. 

Then, by duality between nuclear and operator norms, we compute 
 \begin{equation*}
 \begin{split}
 \bE \left ( Z_T\right )&\leq 8\,c_{\phi}\,\sqrt{l}\,\bE\left (\underset{\left\Vert A \right\Vert_*\leq \sqrt{r\,T/\omega_{\min}}}{\sup}\left \vert \left\langle \Sigma_R,A\right\rangle\right \vert\right )\\&\leq 8\,c_{\phi}\,\sqrt{\dfrac{l\,r\,T}{\omega_{\min}}}\,\bE\left ( \left\Vert \Sigma_R\right\Vert\right ).
 \end{split}
 \end{equation*}
Finally, using 
$$\dfrac{1}{9}\dfrac{5}{12}T+8\,c_{\phi}\,\sqrt{\dfrac{l\,r\,T}{\omega_{\min}}}\,\bE\left ( \left\Vert \Sigma_R\right\Vert\right )\leq \left (\dfrac{1}{9}+\dfrac{8}{9}\right ) \dfrac{5}{12}T+\dfrac{44\,c^{2}_{\phi}\,l\,r}{\omega_{\min}}\left (\bE\left ( \left\Vert \Sigma_R\right\Vert\right )\right )^{2}$$
and the concentration bound \eqref{concentration} we obtain that 
$$\mathbb P\left (Z_T>\dfrac{5}{12}T+ \dfrac{44\,c^{2}_{\phi}\,l\,r}{\omega_{\min}}\left (\bE\left ( \left\Vert \Sigma_R\right\Vert\right )\right )^{2}\right )\leq \exp\left (-\dfrac{c_3nT^{2}}{c^{2}_{\phi}\,l}\right )$$
where $c_3=\dfrac{1}{128}$ as stated.

%% % % % % % % % % % % % % % % % % % % % % % % % % % % % % % % % % % % % % % % % % % % % % % % % %

 \section{Proof of Lemma \ref{l2}}\label{pl2}

Using \eqref{1} we compute
\begin{equation*}
\lambda\left ( \Vert \hat A\Vert_1-\Vert A_0\Vert_1\right )\leq 2\left\Vert \Sigma\right\Vert \Vert H\Vert_1.
\end{equation*}
The condition $\lambda\geq 3\left \Vert \Sigma\right\Vert$, the triangle inequality and \eqref{un2}  yield
 \begin{equation*}
 \begin{split}
 \lambda\left (\left\Vert \mathbf P_{A_0}^{\bot}(H)\right\Vert_1- \left\Vert\mathbf P_{A_0}(H)\right\Vert_1\right ) \leq \dfrac{2}{3}\lambda\left (\left\Vert \mathbf P_{A_0}^{\bot}(H)\right\Vert_1+ \left\Vert\mathbf P_{A_0}(H)\right\Vert_1\right ).
 \end{split}
 \end{equation*}
 This implies that
  \begin{equation*}
  \begin{split}
  \left\Vert \mathbf P_{A_0}^{\bot}(H)\right\Vert_1&\leq  5\left\Vert\mathbf P_{A_0}(H)\right\Vert_1.
  \end{split}
  \end{equation*}  
as stated.
%% % % % % % % % % % % % % % % % % % % % % % % % % % % % % % % % % % % % % % % % % % % % % % % % % % % %

  \section{Bounds on the stochastic errors}\label{stochastic}
  In this section we will obtain upper bounds for the stochastic errors $\left\Vert \Sigma\right\Vert$, $\left\Vert \Sigma_R\right\Vert$. Recall that
  \begin{equation}\label{stoch1}
  \Sigma_R=\dfrac{1}{n} \Sum \epsilon_i X_i\qquad\text{and}\qquad \Sigma=\dfrac{1}{n}\Sum \left (W^{T}_i\rho^{(l)}(t_i)+\sigma\,\xi_i\right )X_i
  \end{equation}
  where $\{\epsilon_i\}_{i=1}^{n}$ is an i.i.d. Rademacher sequence.
  
   The following proposition is the matrix version of Bernstein's inequality in the bounded case (see  Theorem 1.6 in \cite{tropp-user}).
   Let $Z_1,\dots ,Z_n$ be independent random matrices with dimensions $m_1\times m_2$. Define
  \begin{equation*}
  \sigma_Z=\max\left \{\left \Vert \dfrac{1}{n}\Sum \bE\left (Z_iZ^{T}_i\right )\right \Vert^{1/2}, \left \Vert \dfrac{1}{n}\Sum \bE\left (Z_i^{^{T}}Z_i\right )\right \Vert^{1/2}\right \}.
  \end{equation*}
  \begin{proposition}\label{pr_bounded}
     Let $Z_1,\dots ,Z_n$ be independent random matrices with dimensions $m_1\times m_2$ that satisfy $\bE(Z_i)=0$. Suppose that  $\Vert Z_i\Vert\leq U$ for some constant $U$ and all $i=1,\dots,n$. Then, for all $t>0$, with probability at least $1-e^{-t}$ we have
     $$ \left\Vert \dfrac{1}{n}\Sum Z_i\right\Vert\leq 2\max \left \{\sigma_Z\sqrt{\dfrac{t+\log(d)}{n}},U\,\dfrac{t+\log(d)}{n}\right \},$$ 
     where $d=m_1+m_2$.
    \end{proposition}
 It is possible to extend this result to the sub-exponential case. Set
     \begin{equation*}
     U_i=\inf\left \{K>0\;:\;\bE\exp\left (\Vert Z_i\Vert/K\right )\leq e\right \}.
     \end{equation*}
  The following proposition is obtained by an extension of Theorem 4 in \cite{koltchinskii-remark} to rectangular matrices via self-adjoint dilation (cf., for example 2.6 in \cite{tropp-user}).

  \begin{proposition}\label{pr1}
   Let $Z_1,\dots ,Z_n$ be independent random matrices with dimensions $m_1\times m_2$ that satisfy $\bE(Z_i)=0$. Suppose that  $U_i<U$ for some constant $U$ and all $i=1,\dots,n$. Then, there exists an absolute constant $c^{*}$, such that, for all $t>0$, with probability at least $1-e^{-t}$ we have
   $$ \left\Vert \dfrac{1}{n}\Sum Z_i\right\Vert\leq c^{*}\max \left \{\sigma_Z\sqrt{\dfrac{t+\log(d)}{n}},U\left (\log\dfrac{U}{\sigma_Z}\right )\dfrac{t+\log(d)}{n}\right \},$$ 
   where $d=m_1+m_2$.
  \end{proposition}

  We use Propositions \ref{pr_bounded} and \ref{pr1} to prove Lemmas \ref{delta} and \ref{Edelta}.

  %%%%%%%%%%%%%%%%%%%%%%%%%%%%%%%%%%%%%%%%%%%%%%%%%%%%%%%%%%%%%%%%%%%%%%%%%%%%%%%%%%%%%%%%%%%%%%%%
  
 \subsection{Proof of  Lemma \ref{delta}}

 Let $\Sigma_1=\dfrac{1}{n}\Sum W^{T}_i\rho^{(l)}(t_i)X_i$ and $\Sigma_2=\dfrac{1}{n}\Sum \xi_iX_i$. 
Then, we obtain $\Sigma=\Sigma_1+\sigma\,\Sigma_2$. In order to derive an upper bound for $ \left\Vert \Sigma_2\right\Vert$,
we apply Proposition \ref{pr1} to
 \begin{equation*}
 \begin{split}
 Z_i&=\xi_iX_i =\xi_iW_i\phi^{T}(t_i).
 \end{split}
 \end{equation*}  
We need to estimate $\sigma_Z$ and $U$. Note that $Z_i$ is a zero-mean random matrix such that
  \begin{equation*}
  \begin{split}
  \left\Vert Z_i\right\Vert&\leq \vert\xi_i\vert \left\Vert W_i\phi^{T}(t_i)\right\Vert_2 %\\  &
 = \vert\xi_i\vert \left\Vert W_i\phi^{T}(t_i)\right\Vert_2\\ 
  &= \vert\xi_i\vert \left\Vert W_i\right\Vert_2\left\Vert \phi^{T}(t_i)\right\Vert_2 % \\ &
 \leq \vert\xi_i\vert \left\Vert \phi^{T}(t_i)\right\Vert_2\\
  &\leq \vert\xi_i\vert\,c_{\phi}\,\sqrt{l}
  \end{split}
  \end{equation*}}
 where we used condition \eqref{eq:basis_assum} and  $\left\Vert W\right\Vert_2\leq 1$. 
Then, Assumption \ref{noise} implies that there exists a constant $K$ such that $U_i\leq K\,c_{\phi}\,\sqrt{l}$ for all $i=1,\dots,n$.
  
  Let us estimate $\sigma_Z$ for $Z=\xi\,W\phi^{T}(t)$. First we compute $\dfrac{1}{n}\Sum\bE\left (Z_i\,Z_i^{T}\right )$:
  \begin{equation}\label{sigma_Z_1}
  \begin{split}
  \dfrac{1}{n}\Sum\bE\left (Z_iZ^{T}_i\right )&= \dfrac{1}{n}\Sum\bE\left (\xi_i^{2}W_i\phi^{T}(t_i)\phi(t_i)W^{T}_i\right )\\
  &= \bE\left (\left\Vert \phi(t)\right\Vert^{2}_2 W\,W^{T}\right )\\
  %&\leq l\,\bE\left (W\,W^{T}\right )\\
    &=l\,\Omega
  \end{split}
  \end{equation}
  where we used $\bE(\xi^{2})=1$. %, $\left\Vert \phi(t)\right\Vert^{2}_2\leq l$ 
%  and the fact that $W\,W^{T}\geq 0$. We also get
%\begin{equation}\label{low_bound_sigma}
%  \begin{split}
%  \dfrac{1}{n}\Sum\bE\left (Z_iZ^{T}_i\right )
%  &= \bE\left (\left\Vert \phi(t)\right\Vert^{2}_2\bE\left ( W\,W^{T}\vert\,t\right )\right )\\
%  &\geq \Omega\,\bE\left (\left\Vert \phi(t)\right\Vert^{2}_2\right )\\
%    &=l\,\Omega.
%  \end{split}
%  \end{equation}
  
  Now we compute $\dfrac{1}{n}\Sum\bE\left (Z_i^{T}\,Z_i\right )$:
  \begin{equation}\label{sigma_Z_2}
    \begin{split}
    \dfrac{1}{n}\Sum\bE\left (Z^{T}_iZ_i\right )&= \dfrac{1}{n}\Sum\bE\left (\xi_i^{2}\phi(t_i)W^{T}_iW_i\phi^{T}(t_i)\right )\\
    &= \bE\left (\phi(t)\phi^{T}(t) \left\Vert W\right\Vert^{2}_{2}\right )\\
    &=\mathrm{tr}\left (\Omega\right )\,\mathbb{I}_l
    \end{split}
    \end{equation}
    where we used that $(\phi_i(\cdot))_{i=1,\dots,\infty}$ is an orthonormal basis in $L_2\left ((0,1),d\mu\right )$.
    % and $\left\Vert W\right\Vert^{2}_{2}\leq \left\Vert W\right\Vert^{2}_{1}\leq 1$. 
    
     Equations
\eqref{sigma_Z_1} and \eqref{sigma_Z_2} imply that
  $$\sigma_Z^{2}\leq \left (l\,\omega_{\max}\right )\vee \mathrm{tr}\left (\Omega\right )\quad\text{and}\quad\sigma_Z^{2}\geq l\,\omega_{\max}.$$ Applying Proposition \ref{pr1} we derive that
   for all $t>0$  with probability at least $1-e^{-t}$
  \begin{equation}\label{sigma_2}
 \begin{split}
 \left\Vert \Sigma_2\right\Vert&\leq c^{*}\max \left \{\sqrt{\dfrac{M\,\left (t+\log(d)\right )}{n}},\dfrac{K\,c_{\phi}\,\sqrt{l}\,\left (t+\log(d)\right )\,\log\left (\dfrac{K\,c_{\phi}}{\omega_{\max}}\right )}{n}\right \}
\end{split}
\end{equation}
where $M=\mathrm{tr}(\Omega)\vee \left (l\omega_{\max}\right )$.

  One can estimate  
$ \left\Vert \Sigma_1\right\Vert$ in a similar way.
  We apply Proposition \ref{pr_bounded} to
   \begin{equation*}
  \begin{split}
  Z_i&=W^{T}_i\rho^{(l)}(t_i)X_i\\&=W^{T}_i\rho^{(l)}(t_i)W_i\phi^{T}(t_i).
  \end{split}
  \end{equation*} We begin by proving that $$\bE\left (W^{T}\rho^{(l)}(t)W\phi^{T}(t)\right )=0.$$  Let $W=\left (w_1,\dots,w_p\right )$. The $(m,k)$-th entry of $W^{T}\rho^{(l)}(t)W\phi^{T}(t)$ is equal to $\underset{j=1}{\overset{p}{\Sigma}}w_j\,\rho_j^{(l)}(t)\,w_m\,\phi_k(t)$. By definition $\rho^{(l)}_j(t)=f_j(t)-\underset{i=1}{\overset{l}{\Sigma}}a^{0}_{ji}\phi_i(t)$ and we compute
  \begin{equation*}
    \begin{split}
    \bE\left (\rho^{(l)}_j(t)\phi_k(t)\right )&=\bE\left (\left (f_j(t)-\underset{i=1}{\overset{l}{\Sigma}}a^{0}_{ji}\phi_i(t)\right )\phi_k(t)\right )\\&=\bE\left (f_j(t)\phi_k(t)-\underset{i=1}{\overset{l}{\Sigma}}a^{0}_{ji}\phi_i(t)\phi_k(t)\right )
    \\&=a^{0}_{jk}-a^{0}_{jk}=0
    \end{split}
    \end{equation*}
    since % where we used that 
$(\phi_i(\cdot))_{i=1,\dots,\infty}$ is an orthonormal basis. Therefore, %This implies 
     \begin{equation*}
     \begin{split}
     \bE\left (\underset{j=1}{\overset{p}{\Sigma}}w_j\,\rho_j^{(l)}(t)\,w_m\,\phi_k(t)\right )&=\underset{j=1}{\overset{p}{\Sigma}}\bE_{W}\left (w_j\,w_m\bE_{t}\left (\rho_j^{(l)}(t)\,\phi_k(t)\right )\right )\\&=0
     \end{split}
     \end{equation*}
      
 Next we estimate $U$.  Note that $\rho^{(l)}(t)$ has at most $s-1$ % Should not this be s-1? $s$ 
non-zero coefficients. Then, Assumption \ref{approximation_sup} and $ \left\Vert W\right\Vert_{2}\leq 1$ imply that $t-$almost surely \\$\left (W^{T}\rho^{(l)}(t)\right )^{2}\leq \dfrac{b^{2}\,(s-1)}{l^{2\gamma}}$ and
  \begin{equation*}
    \begin{split}
    \left\Vert Z_i\right\Vert&\leq \vert W^{T}_i\rho^{(l)}(t_i)\vert \left\Vert W_i\phi^{T}(t_i)\right\Vert\\
    &\leq \dfrac{b\,c_{\phi}\,\sqrt{l\,(s-1)}}{l^{\gamma}}.
    \end{split}
    \end{equation*}
  Let us estimate $\sigma_Z$ for $Z=\left (W^{T}\rho^{(l)}(t)\right )W\phi^{T}(t)$. First we compute $\dfrac{1}{n}\Sum\bE\left (Z_i\,Z_i^{T}\right )$:
    \begin{equation*}
    \begin{split}
    \dfrac{1}{n}\Sum\bE\left (Z_iZ^{T}_i\right )&= \bE\left (\left (W^{T}\rho^{(l)}(t)\right )^{2}W\phi^{T}(t)\phi(t)W^{T}\right )\\
    &= \bE_t\left (\left \Vert \phi(t)\right\Vert^{2}_2\,\bE_W\left ( \left (W^{T}\rho^{(l)}(t)\right )^{2}WW^{T}\right )\right ).
    \end{split}
    \end{equation*}
 We obtain %have
  \begin{equation*}
  \begin{split}
  \bE_W\left ( \left (W^{T}\rho^{(l)}(t)\right )^{2}WW^{T}\right )\leq \dfrac{b^{2}\,(s-1)}{l^{2\gamma}}\,\bE\left (WW^{T}\right )
  \end{split}
  \end{equation*}
  where we used $WW^{T}\geq 0$. Finally we obtain
   \begin{equation}\label{sigma_Z_1_2}
 \left\Vert \dfrac{1}{n}\Sum\bE\left (Z_iZ^{T}_i\right ) \right\Vert\leq \dfrac{b^{2}\,(s-1)\,\omega_{\max}\,l}{l^{2\gamma}}.
\end{equation}

  Now we compute $\dfrac{1}{n}\Sum\bE\left (Z_i^{T}\,Z_i\right )$:
  \begin{equation*}\label{sigma_Z_2_2}
      \begin{split}
      \dfrac{1}{n}\Sum\bE\left (Z^{T}_iZ_i\right )
      &= \bE_t\left (\left ( W^{T}\rho^{(l)}(t)\right )^{2}\phi(t)W^{T}W\phi^{T}(t)\right )\\
      &= \bE_t\left (\left ( W^{T}\rho^{(l)}(t)\right )^{2} \left\Vert W\right\Vert^{2}_{2}\phi(t)\phi^{T}(t)\right ).
           \end{split}
      \end{equation*}
      Using $\bE\left (\left\Vert W\right\Vert^{2}_{2}\right )=\mathrm{tr}(\Omega)$ and $\bE_t\left (\phi(t)\phi^{T}(t)\right )=\mathbb{I}_{l}$ we obtain
      \begin{equation}\label{estim_sigma_Z_3_aout}
       \left\Vert \dfrac{1}{n}\Sum\bE\left (Z^{T}_iZ_i\right ) \right\Vert\leq \dfrac{b^{2}\,(s-1)}{l^{2\gamma}}\,\mathrm{tr}(\Omega).
      \end{equation}
Equations 
  \eqref{sigma_Z_1_2} and \eqref {estim_sigma_Z_3_aout} imply that 
    $$\sigma^{2}_Z\leq \dfrac{b^{2}\,(s-1)}{l^{2\gamma}}\left [\mathrm{tr}(\Omega)\vee (l\,\omega_{\max})\right ].$$

      Applying Proposition \ref{pr_bounded}, we derive that
     for all $t>0$  with probability at least $1-e^{-t}$
     \begin{equation} \label{sigma_1}
     \left\Vert \Sigma_1\right\Vert\leq \dfrac{2\,b\,\sqrt{s-1}}{l^{\gamma}}\max \left \{\sqrt{\dfrac{M(t+\log(d))}{n}},\dfrac{c_{\phi}\,\sqrt{l}\left (t+\log(d)\right )}{n}\right \}.
     \end{equation}
The bounds \eqref{sigma_1} and \eqref{sigma_2} imply  that % there exists a numerical constant $C^{*}$, that depends only on $K$, such that,
    for all $t>0$  with probability at least $1-2e^{-t}$
     \begin{equation*}
     \begin{split}
     \left\Vert \Sigma\right\Vert&\leq \left (\sigma\,c^{*}+\dfrac{2\,b\,\sqrt{s-1}}{l^{\gamma}}\right )\max \left \{\sqrt{\dfrac{M\left (t+\log(d)\right )}{n}},\right. \\&\left .\hskip 4 cm\dfrac{c_{\phi}\,\sqrt{l}\,\left (t+\log(d)\right )\left (\left [K\,\log\left (\dfrac{K}{\omega_{\max}}\right )\right ]\vee 1\right )}{n}\right \}
     \end{split}
     \end{equation*}
      as stated.

%% % % % % % % % % % % % % % % % % % % % % % % % % % % % % % % % % % % % % % % % % % % % % % % % % % % % % % % % % % % % % % %
%%%%%%%%%%%%%%%%%%%%%%%%%%%%%%%%%%%%%%%%%%%%%%%%%%%%%%%%%%%%%%%%%%%%%%%%%

\subsection{ Proof of Lemma  \ref{Edelta}}\label{proof_E}

The proof follows the lines of the proof of Lemma 7 in \cite{klopp-rank}. 
We use Proposition~\ref{pr_bounded} with $Z_i=\epsilon_i\,X_i$. As in the proof of Lemma \ref{delta},  we obtain $U=\sqrt{l}$ and $\sigma^{2}_{Z}=\left (\mathrm{tr}(\Omega)\vee \left (l\sigma_{\max}(\Omega)\right )\right )$. Set $M=\left (\mathrm{tr}(\Omega)\vee \left (l\sigma_{\max}(\Omega)\right )\right )$, then Proposition \ref{pr_bounded} implies that for all $t>0$ with probability at least $1-e^{-t}$
\begin{equation}\label{est_norm_sigma_R}
\left\Vert \Sigma_R\right\Vert\leq 2\max \left \{\sqrt{\dfrac{M\left (t+\log(d)\right )}{n}},\dfrac{\sqrt{l}\left (t+\log(d)\right )}{n}\right \}.
\end{equation}
 Set $t^{*}=\dfrac{n\,M}{l}-\log (d)$ so that  $t^{*}$ is the value of $t$ such that the two terms in \eqref{est_norm_sigma_R} are equal. 
Note that \eqref{est_norm_sigma_R} implies that 
\begin{equation}\label{proba_1}
\mathbb P \left (\left\Vert \Sigma_R\right\Vert> t\right )\leq d\exp\left \{-\dfrac{t^{2}\,n}{4\,M}\right \}\qquad\text{for}\qquad t\leq t^{*}
\end{equation}
 and
\begin{equation}\label{proba_2}
\mathbb P \left (\left\Vert \Sigma_R\right\Vert> t\right )\leq d\exp\left \{-\dfrac{t\,n}{2\,\sqrt{l}}\right \}\qquad\text{for}\qquad t\geq t^{*}.
\end{equation}
We set $\nu_1=\dfrac{n}{4\,M}$, $\nu_2=\dfrac{n}{2\,\sqrt{l}}$. By H{\"o}lder's inequality we derive $$\bE\left\Vert \Sigma_R\right\Vert\leq \left (\bE \left\Vert \Sigma_R\right\Vert^{2\log( d)}\right )^{1/(2\log (d))}.$$
Inequalities \eqref{proba_1} and \eqref{proba_2} imply that 
\begin{equation}\label{estEM}
\begin{split}
&\left (\bE \left\Vert \Sigma_R\right\Vert^{2\log (d)}\right )^{1/2\log (d)} =\left ( \overset{+\infty}{\underset{0}{\int}}\mathbb P \left (\left\Vert \Sigma_R\right\Vert> t^{1/(2\log (d))}\right )\mathrm{d}t\right )^{1/2\log (d)}\\&\hskip 0.5 cm\leq
\left (d \overset{+\infty}{\underset{0}{\int}}\exp\{-t^{1/\log (d)}\nu_1\}\mathrm{d}t+d \overset{+\infty}{\underset{0}{\int}}\exp\{-t^{1/(2\log (d)}\nu_2\}\mathrm{d}t
\right )^{1/2\log (d)}\\&\hskip 1 cm\leq \sqrt{e}\left (\log (d)\nu_1^{-\log (d)}\Gamma(\log (d))+2\log(d)\; \nu_2^{-2\log (d)}\Gamma(2\log (d))\right )^{1/(2\log (d))}.
\end{split}
\end{equation}
Recall that Gamma-function satisfies the following inequality
% The  bound Gamma-function satisfies the following: 
\begin{equation}\label{Gamma}
% \text{for}\quad x\geq 2,\quad\Gamma(x)\leq \left (\dfrac{x}{2}\right )^{x-1}
\Gamma(x)\leq \left (\dfrac{x}{2}\right )^{x-1}   \quad \text{for}\quad x\geq 2, 
\end{equation} 
(see e.g. \cite{klopp-rank}). Plugging \eqref{Gamma} into \eqref{estEM} we compute  
\begin{equation*}
\begin{split}
\bE \left\Vert \Sigma_R\right\Vert& \leq \sqrt{e}\left ((\log (d))^{\log (d)}\nu_1^{-\log (d)}2^{1-\log (d)}\right .\\&+\left .2(\log (d))^{2\log (d)}\nu_2^{-2\log (d)}\right )^{1/(2\log (d))}. 
\end{split}
\end{equation*}
Observe that $n\geq n^{*}$ implies 
$\nu_1\log (d)\leq\nu_2^{2}$ and we obtain
 \begin{equation}\label{concl_esp}
 \begin{split}
 \bE \left\Vert \Sigma_R\right\Vert& \leq \sqrt{\dfrac{2e\log (d)}{\nu_1}}.
 \end{split}
 \end{equation}
We conclude the proof by plugging $\nu_1=\dfrac{n}{4\,M}$ into \eqref{concl_esp}.

% % % % % % % % % % % % % % % % % % % % % % % % % % % % % % % % % % % % % % % % % % % % % % % % % % % % % % % %

%
%
\end{document}